\documentclass[reqno]{amsart}
\usepackage{amssymb}
\usepackage{amsmath}
\usepackage[all]{xy}
\newtheorem{theorem}{Theorem}[section]
\newtheorem{lemma}[theorem]{Lemma}
\newtheorem{proposition}[theorem]{Proposition}

\theoremstyle{definition}
\newtheorem{definition}[theorem]{Definition}

\newtheorem{remark}[theorem]{Remark}
\numberwithin{equation}{section}

\title{Analysis of $U_q(sl(m+1))$-symmetries on quantum $n$-spaces}
\author{Steven Duplij}
\address{Theory Group, Nuclear Physics Laboratory, V. N. Karazin Kharkov National University, 4 Svoboda Sq., Kharkov, 61022, Ukraine} 
\email{duplij@math.rutgers.edu, sduplij@gmail.com}
\urladdr{http://www.math.rutgers.edu/~duplij}
\author{Yanyong Hong}
\address{Department of Mathematics, Zhejiang University,
Hangzhou, 310027, P.R.China;\;\; The corresponding author} \email{hongyanyong2008@zju.edu.cn}
\author{Fang Li}
\address{Department of Mathematics, Zhejiang University,
Hangzhou, 310027, P.R.China} \email{fangli@zju.edu.cn}
\begin{document}

\begin{abstract}
In this paper, the module-algebra structures of $U_q(sl(m+1))$ on the quantum $n$-space $A_q(n)$ are studied. 
We characterize all module-algebra structures of $U_q(sl(m+1))$ on $A_q(2)$ and $A_q(3)$ when $m\geq 2$. 
The module-algebra structures of $U_q(sl(m+1))$ on $A_q(n)$ are also considered for any $n\geq 4$.
\end{abstract}
\maketitle

\section{Introduction}
The actions of Hopf algebras \cite{montgomery} and their generalizations (see, e.g., \cite{dra/van/zha}) play an important role in the quantum group theory
\cite{kassel,KlSch} and in its various application in physics \cite{cas/wes}. However, 
it was long time believed that the quantum plane \cite{manin2}
admits only one special symmetry \cite{mon/smi} inspired by the action of
$U_{q}( sl(  2) )  $ (in other words the $U_{q}(
sl(  2) )  $-module algebra structure \cite{kassel}). In \cite{hu2000},
the quantum $n$-space is equipped with the special
$U_q(sl(m+1))$-module algebra structure via a certain $q$-differential
operator realization.
Then it was shown \cite{dup/sin3}, that the $U_{q}(  sl(  2)
)  $-module algebra structure on the quantum plane is much more diverse
and consists of 8 nonisomorphic cases \cite{dup/sin3,dup/sin4}. The full
classification was given in terms of the introduced so-called weight. Its
introduction follows from the general form of an automorphism of the quantum
plane \cite{ale/cha}. Some properties of the actions of commutative Hopf
algebras on quantum polynomials were studied in \cite{art2000,art2002}.

Here we consider the actions of
quantum universal enveloping algebra $U_{q}(sl(m+1))$ on the quantum $n$-space
$A_{q}(  n)  $. We use the method of weights
\cite{dup/sin3,dup/sin4} to classify some actions in terms of the introduced
action matrices. Then we present the Dynkin diagrams for the obtained actions
and find their classical limit. A special case discussed in this paper was also included in \cite{hu2000}.

This work is organized as follows. In Section 2, we give the necessary
preliminary information and notation, as well as prove an important lemma about
action on generators and any elements of the quantum $n$-space. In Section 3,
we present a general idea about how to connect actions of $U_{q}( sl(
2))$ with those of $U_{q}(  sl(  m+1))  $ and characterize all the module-algebra structures of $U_q(sl(m+1))$ on the quantum plane when $m\geq 2$.  Section 4 is devoted to
classification of  $U_{q}(  sl(  m+1))  $ actions
on the quantum 3-space $A_{q}( 3) $ using the method of
weights \cite{dup/sin3,dup/sin4} for $m\geq 2$. Then we present the classical limit of the
obtained actions together with the Dynkin diagrams. In Section 5, we study the module-algebra structures of $U_q(sl(m+1))$ on the quantum $n$-space for $n\geq 4$.

\section{Preliminaries}

In this paper, all algebras, modules and linear spaces are over the field $\mathbb{C}$ of complex numbers.

Let $H$ be a Hopf algebra whose comultiplication is $\Delta$, counit is $\varepsilon$ and antipode is $S$ and let $A$ be a unital algebra with unit $1$.
Using the Sweedler notation, we set $\Delta(h)=\sum_i h_i^{'}\otimes h_i^{''}$.
\begin{definition} By a structure of $H$-\emph{module algebra} (or say, $H$-{\em symmetry}) on $A$, we mean a homomorphism $\pi: H\rightarrow End_{\mathbb{C}}A$ such that:\\
(1) $\pi(h)(ab)=\sum_i\pi(h_i^{'})(a)\pi(h_i^{''})(b)$ for all $h\in H$, $a$, $b\in A$,\\
(2) $\pi(h)(1)=\varepsilon(h)1$ for all $h\in H$.

The structures $\pi_1$, $\pi_2$ are said to be \emph{isomorphic}, if there exists an automorphism $\psi$ of $A$ such that $\psi\pi_1(h)\psi^{-1}=\pi_2(h)$ for all $h\in H$.
\end{definition}

Throughout this paper we assume that $q\in \mathbb{C}\backslash\{0\}$ and $q$ is not a root of unity.

For any integer $n$, we introduce the $q$-integer by $(n)_q=\frac{q^n-1}{q-1}.$
They were introduced by Heine \cite{heine} and are called the Heine numbers or
$q$-deformed numbers \cite{kac/che}. If $n>0$,
$$(n)_q=\frac{q^n-1}{q-1}=1+q+\cdots+q^{n-1}.$$

First, we introduce the definition of $U_q(sl(m+1))$.
\begin{definition} The quantum universal enveloping algebra $U_q(sl(m+1))$ $(m\geq 1)$ as the algebra is generated by $(e_i,f_i,k_i,k_i^{-1})_{1\leq i\leq m}$ and the relations
\begin{eqnarray}\label{w5}
k_ik_i^{-1}=k_i^{-1}k_i=1,~~k_ik_j=k_jk_i,\end{eqnarray}
\begin{eqnarray}\label{w2}
k_ie_jk_i^{-1}=q^{a_{ij}}e_j,~~k_if_jk_i^{-1}=q^{-a_{ij}}f_j, \end{eqnarray}
\begin{eqnarray}
[e_i, f_j]=\delta_{ij}\frac{k_i-k_i^{-1}}{q-q^{-1}},
\end{eqnarray}
\begin{eqnarray}
e_ie_j=e_je_i\qquad\text{and}\qquad f_if_j=f_jf_i,\qquad \text{if~~$a_{ij}=0$,}
\end{eqnarray}
if $a_{ij}=-1$,
\begin{eqnarray}
e_i^2e_j-(q+q^{-1})e_ie_je_i+e_je_i^2=0,\end{eqnarray}
\begin{eqnarray}\label{w6}
f_i^2f_j-(q+q^{-1})f_if_jf_i+f_jf_i^2=0,
\end{eqnarray}
where for any $i$, $j\in\{1,2,\cdots, m\}$, $a_{ii}=2$ and $a_{ij}=0$, if $|i-j|>1$; $a_{ij}=-1$, if $|i-j|=1$.
\end{definition}

The standard Hopf algebra structure on $U_q(sl(m+1))$ is determined by
\begin{eqnarray} \Delta(e_i)=1\otimes e_i+e_i\otimes
k_i,\end{eqnarray}
\begin{eqnarray} \Delta(f_i)=k_i^{-1}\otimes f_i+f_i\otimes
1,\end{eqnarray}
\begin{eqnarray} \Delta(k_i)=k_i\otimes k_i,\qquad \Delta(k_i^{-1})=k_i^{-1}\otimes k_i^{-1},\end{eqnarray}
\begin{eqnarray} \varepsilon(k_i)=\varepsilon(k_i^{-1})=1,\end{eqnarray}
\begin{eqnarray} \varepsilon(e_i)=\varepsilon(f_i)=0,\end{eqnarray}
\begin{eqnarray}S(e_i)=-e_ik_i^{-1},\qquad S(f_i)=-k_if_i,\end{eqnarray}
\begin{eqnarray} S(k_i)=k_i^{-1},\qquad S(k_i^{-1})=k_i,\end{eqnarray}
for $i\in \{1,2,\cdots, m\}$.

Next, we introduce the definition of quantum $n$-space
(see \cite{goo/let,art2000}).
\begin{definition}
 The \emph{quantum $n$-space} $A_q(n)$
 is a unital algebra generated by $n$ generators $x_i$ for $i\in\{1,\cdots, n\}$ satisfying the relations
\begin{eqnarray}\label{e61}x_ix_j=qx_jx_i~~~for ~~any ~~i>j.\end{eqnarray}
\end{definition}

If $n=2$, $A_q(2)$ is called  a \emph{quantum plane} (see \cite{kassel}).
In this case, Duplij and Sinel'shchikov studied the classification of $U_q(sl(2))$-module algebra structures on $A_q(2)$ in \cite{dup/sin3}.

The next proposition is very important for us to study the module-algebra structures
of $U_q(sl(m+1))$ on $A_q(n)$ (see \cite{ale/cha,art2006}).
\begin{proposition}\label{w3}
Let $\varphi$ be an automorphism of $A_q(n)$, then there exist nonzero constants $\alpha_i\in \mathbb C$ for $i\in \{1,\cdots, n\}$ such that
$$\varphi:~~x_i\rightarrow \alpha_i x_i.$$ \end{proposition}

Finally, we present a lemma which is useful for us in checking the module-algebra structures of $U_q(sl(m+1))$ on $A_q(n)$.
\begin{lemma}\label{w8}
Given the module-algebra actions of the generators $k_i$, $e_i$, $f_i$ of $U_q(sl(m+1))$ on $A_q(n)$ for $i\in\{1,\cdots,m\}$, if an element in the ideal from the relations (\ref{w5})-(\ref{w6}) of $U_q(sl(m+1))$  acting on the generators $x_i$ of $A_q(n)$ produces zero for $i\in\{1,\cdots,n\}$, then this element acting on any $v\in A_q(n)$ produces zero.
\end{lemma}
\begin{proof}
Here, we only prove that if $e_i^2e_{i+1}(x)-(q+q^{-1})e_ie_{i+1}e_i(x)+e_{i+1}e_i^2(x)=0$ and $e_i^2e_{i+1}(y)-(q+q^{-1})e_ie_{i+1}e_i(y)+e_{i+1}e_i^2(y)=0$,
then $e_i^2e_{i+1}(xy)-(q+q^{-1})e_ie_{i+1}e_i(xy)+e_{i+1}e_i^2(xy)=0$ where $x,y$ are both generators of $A_q(n)$. The other relations can be proved similarly. \\
$e_i^2e_{i+1}(xy)-(q+q^{-1})e_ie_{i+1}e_i(xy)+e_{i+1}e_i^2(xy)$
\begin{eqnarray*}
&=&e_i^2(xe_{i+1}(y)+e_{i+1}(x)k_{i+1}(y))-(q+q^{-1})e_ie_{i+1}(xe_i(y)+e_i(x)k_i(y))\\
&&+e_{i+1}e_i(xe_i(y)+e_i(x)k_i(y))\\
&=&e_i(xe_ie_{i+1}(y)+e_i(x)k_ie_{i+1}(y)+e_{i+1}(x)e_ik_{i+1}(y)+e_ie_{i+1}(x)k_ik_{i+1}(y))\\
&&-(q+q^{-1})e_i(xe_{i+1}e_i(y)+e_{i+1}(x)k_{i+1}e_i(y)+e_i(x)e_{i+1}k_i(y)+e_{i+1}e_i(x)\\
&&\cdot k_{i+1}k_i(y))+e_{i+1}(xe_i^2(y)+e_i(x)k_ie_i(y)+e_i(x)e_ik_i(y)+e_i^2(x)k_i^2(y))\\
&=&(xe_i^2e_{i+1}(y)-(q+q^{-1})xe_ie_{i+1}e_i(y)+xe_{i+1}e_i^2(y))+(e_i^2e_{i+1}(x)-(q+q^{-1})\\
&&\cdot e_ie_je_i(x)+e_je_i^2(x))k_i^2k_{i+1}(y)+(e_i(x)k_ie_ie_{i+1}(y)+e_i(x)e_ik_ie_{i+1}(y)\\
&&-(q+q^{-1})e_i(x)e_ie_{i+1}k_i(y))+(e_i^2(x)k_i^2e_{i+1}(y)-(q+q^{-1})e_i^2(x)\\
&&\cdot k_ie_{i+1}k_i(y)+e_i^2(x)e_{i+1}k_i^2(y))+(e_{i+1}(x)e_i^2k_{i+1}(y)-(q+q^{-1})e_{i+1}(x)\\
&&\cdot e_ik_{i+1}e_i(y)+e_{i+1}(x)k_{i+1}e_i^2(y))+(e_ie_{i+1}(x)k_ie_ik_{i+1}(y)+e_ie_{i+1}(x)\\
&&\cdot e_ik_ik_{i+1}(y)-(q+q^{-1})e_ie_{i+1}(x)k_ik_{i+1}e_i(y))+(-(q+q^{-1})e_{i+1}e_i(x)\\
&&\cdot e_ik_{i+1}k_i(y)+e_{i+1}e_i(x)k_{i+1}k_ie_i(y)+e_{i+1}e_i(x)k_{i+1}e_ik_i(y))\\
&&+(-(q+q^{-1})e_i(x)k_ie_{i+1}e_i(y)+e_i(x)e_{i+1}k_ie_i(y)+e_i(x)e_{i+1}e_ik_i(y))\\
&&=0.
\end{eqnarray*}
Thus, the lemma holds.
\end{proof}

Therefore, by Lemma \ref{w8}, in checking, whether the relations of $U_q(sl(m+1))$, acting on any $v\in A_q(n)$, produce zero, we only need to check, whether they, acting on the generators $x_1$, $\cdots$, $x_n$, produce zero.
\section{Classification of $U_q(sl(m+1))$-symmetries on $A_q(2)$}
In this section, we study the module-algebra structures of $U_q(sl(m+1))$ on $A_q(2)$ where $m\geq 2$.

According to the paper \cite{dup/sin3} by S.Duplij and S.Sinel'shchikov, we have known all module-algebra structures of $U_q(sl(2))$ on $A_q(2)$. For $U_q(sl(3))$, there are two sub-Hopf algebras which are isomorphic to $U_q(sl(2))$. One is generated by $k_1$, $e_1$ and $f_1$. Denote this algebra by $A$. The other denoted by $B$ is generated by $k_2$, $e_2$ and $f_2$. By the definition of module algebra of one Hopf algebra, the module-algebra structures on $A_q(2)$ of these two sub-Hopf algebras are kinds of those in \cite{dup/sin3}. All kinds of
module-algebra structures
of $A$ and $B$ on $A_q(2)$ are as follows (here, $i=1, 2$):\\
\begin{enumerate}
\item
\begin{eqnarray*}
&&k_i(x_1)=\pm x_1,\quad k_i(x_2)=\pm x_2,\\
&&e_i(x_1)=e_i(x_2)=f_i(x_1)=f_i(x_2)=0.
\end{eqnarray*}
\item
\begin{eqnarray*}
&&k_i(x_1)=qx_1, \quad k_i(x_2)=q^{-1}x_2,\\
&&e_i(x_1)=0,\quad e_i(x_2)=\tau_ix_1,\\
&&f_i(x_1)=\tau_i^{-1}x_2, \quad f_i(x_2)=0,
\end{eqnarray*}
where $\tau_i\in\mathbb{C}\backslash\{0\}$.
\item
\begin{eqnarray*}
&&k_i(x_1)=qx_1, \quad k_i(x_2)=q^{-2}x_2,\\
&&e_i(x_1)=0, \quad e_i(x_2)=b_i,\\
&&f_i(x_1)=b_i^{-1}x_1x_2, \quad f_i(x_2)=-qb_i^{-1}x_2^2,
\end{eqnarray*}
where $b_i\in\mathbb{C}\backslash\{0\}$.
\item
\begin{eqnarray*}
&&k_i(x_1)=q^2x_1, \quad k_i(x_2)=q^{-1}x_2,\\
&&e_i(x_1)=-qc_i^{-1}x_1^2, \quad e_i(x_2)=c_i^{-1}x_1x_2,\\
&&f_i(x_1)=c_i, \quad f_i(x_2)=0,
\end{eqnarray*}
where $c_i\in\mathbb{C}\backslash\{0\}$.
\item
\begin{eqnarray*}
&&k_i(x_1)=q^{-2}x_1, \quad k_i(x_2)=q^{-1}x_2,\\
&&e_i(x_1)=a_i,\quad e_i(x_2)=0,\\
&&f_i(x_1)=-qa_i^{-1}x_1^2+t_ix_2^4, \quad f_i(x_1)=-qa_i^{-1}x_1x_2+s_ix_2^3,
\end{eqnarray*}
where $a_i\in\mathbb{C}\backslash\{0\}$, $t_i$, $s_i\in \mathbb{C}$.
\item
\begin{eqnarray*}
&& k_i(x_1)=qx_1, \quad k_i(x_2)=q^{2}x_2,\\
&& e_i(x_1)=-qd_i^{-1}x_1x_2+u_ix_1^3, \quad e_i(x_2)=-qd_i^{-1}x_2^2+v_ix_1^4,\\
&& f_i(x_1)=0, \quad f_i(x_2)=d_i,
\end{eqnarray*}
where $d_i\in\mathbb{C}\backslash\{0\}$, $u_i$, $v_i\in \mathbb{C}$.
\end{enumerate}

In addition, the associated classical
limit actions of $sl_3$
(here it is the Lie algebra generated
by $h_1$, $h_2$, $e_1$, $e_2$, $f_1$, $f_2$ with
the relations $[e_1,f_1]=h_1$, $[e_2,f_2]=h_2$, $[e_1,f_2]=[e_2,f_1]=0$, $[h_1,e_1]=2e_1$, $[h_1,e_2]=-e_2$, $[h_2,e_2]=2e_2$, $[h_2,e_1]=-e_1$, $[h_1,f_1]=-2f_1$, $[h_1,f_2]=f_2$, $[h_2,f_2]=-2f_2$, $[h_2,f_1]=f_1$, $[h_1,h_2]=0$) on $\mathbb{C}[x_1,x_2]$ by differentiations are derived from the quantum actions via substituting $k_1=q^{h_1}$, $k_2=q^{h_2}$ with subsequent formal passage to the limit as $q\rightarrow 1$.

\begin{theorem}\label{T8}
$U_q(sl(3))$-symmetries on $A_q(2)$ up to isomorphisms
and their classical limits are as follows
\medskip

\begin{tabular}{|l|l|}
\hline $U_q(sl(3))$-symmetries &  Classical limit \\
& $sl_3$-actions on $\mathbb{C}[x_1,x_2]$ \\
\hline
$k_i(x_1)=\pm x_1$, $k_i(x_2)=\pm x_2$, & $h_i(x_1)=0$, $h_i(x_2)=0$,\\
$k_j(x_1)=\pm x_1$, $k_j(x_2)=\pm x_2$, & $h_j(x_1)=0$, $h_j(x_2)=0$, \\
 $e_i(x_1)=e_i(x_2)=f_i(x_1)=f_i(x_2)=0$,  & $e_i(x_1)=e_i(x_2)=f_i(x_1)=f_i(x_2)=0$, \\
$e_j(x_1)=e_j(x_2)=f_j(x_1)=f_j(x_2)=0$ & $e_j(x_1)=e_j(x_2)=f_j(x_1)=f_j(x_2)=0$ \\
\hline
$k_i(x_1)=q^{-2}x_1$, $k_i(x_2)=q^{-1}x_2$, & $h_i(x_1)=-2x_1$, $h_i(x_2)=-x_2$, \\
$k_j(x_1)=qx_1$, $k_j(x_2)=q^{-1}x_2$,& $h_j(x_1)=x_1$, $h_j(x_2)=-x_2$,\\
$e_i(x_1)=1$, $e_i(x_2)=0$, & $e_i(x_1)=1$, $e_i(x_2)=0$, \\
$e_j(x_1)=0$, $e_j(x_2)=x_1$, & $e_j(x_1)=0$, $e_j(x_2)=x_1$,  \\
$f_i(x_1)=-qx_1^2$,  $f_i(x_2)=-qx_1x_2$, & $f_i(x_1)=-x^2$,  $f_i(x_2)=-x_1x_2$,  \\
$f_j(x_1)=x_2$, $f_j(x_2)=0$ & $f_j(x_1)=x_2$, $f_j(x_2)=0$\\
\hline
$k_i(x_1)=qx_1$, $k_i(x_2)=q^{2}x_2$, & $h_i(x_1)=x_1$, $h_i(x_2)=2x_2$,  \\
$k_j(x_1)=qx_1$, $k_j(x_2)=q^{-1}x_2$, & $h_j(x_1)=x_1$, $h_j(x_2)=-x_2$,  \\
$e_i(x_1)=-qx_1x_2$, $e_i(x_2)=-qx_2^2$, & $e_i(x_1)=-x_1x_2$, $e_i(x_2)=-x_2^2$, \\
$e_j(x_1)=0$, $ e_j(x_2)=x_1$, & $e_j(x_1)=0$, $ e_j(x_2)=x_1$,  \\
$f_i(x_1)=0$, $f_i(x_2)=1,$ & $f_i(x_1)=0$, $f_i(x_2)=1,$  \\
 $f_j(x_1)=x_2$, $f_j(x_2)=0$ & $f_j(x_1)=x_2$, $f_j(x_2)=0$\\
\hline
$k_i(x_1)=q^{-2}x_1$, $k_i(x_2)=q^{-1}x_2$, & $h_i(x_1)=-2x_1$, $h_i(x_2)=-x_2$,\\
$k_j(x_1)=qx_1$, $k_j(x_2)=q^{2}x_2$, & $h_j(x_1)=x_1$, $h_j(x_2)=2x_2$, \\
$e_i(x_1)=1$, $ e_i(x_2)=0,$ & $ e_i(x_1)=1$, $ e_i(x_2)=0,$\\
$ e_j(x_1)=-qx_1x_2$, $e_j(x_2)=-qx_2^2,$ & $e_j(x_1)=-x_1x_2$, $e_j(x_2)=-x_2^2,$ \\
$f_i(x_1)=-qx_1^2$, $f_i(x_2)=-qx_1x_2,$ & $f_i(x_1)=-x_1^2$, $f_i(x_2)=-x_1x_2,$ \\
$f_j(x_1)=0$, $f_j(x_2)=1$ & $f_j(x_1)=0$, $f_j(x_2)=1$\\
\hline
\end{tabular}\\
\medskip

\noindent for any $i=1$, $j=2$ or $i=2$, $j=1$. Additively, there are no isomorphisms between these four kinds of module-algebra structures.
\end{theorem}
\begin{proof}
Denote the six cases of module-algebra structures of $A$ (resp. $B$) on $A_q(2)$ by $(A1)$, $\cdots$, $(A6)$ (resp. by $(B1)$, $\cdots$, $(B6)$) respectively.
Then, for studying the module-algebra structures of $U_q(sl(3))$ on $A_q(2)$, we only need to check that whether the actions of $A$ and those of $B$ are compatible. In other words, we only need to check that
\begin{eqnarray}\label{e11}
k_1e_2(u)=q^{-1}e_2k_1(u), \quad k_1f_2(u)=qf_2k_1(u),\end{eqnarray}
\begin{eqnarray}\label{r10}k_2e_1(u)=q^{-1}e_1k_2(u), \quad k_2f_1(u)=qf_1k_2(u),\end{eqnarray}
\begin{eqnarray}\label{r11}e_1f_2(u)=f_2e_1(u),\quad e_2f_1(u)=f_1e_2(u),\end{eqnarray}
\begin{eqnarray}\label{r12}e_1^2e_2(u)-(q+q^{-1})e_1e_2e_1(u)+e_2e_1^2(u)=0,\end{eqnarray}
\begin{eqnarray}\label{r13}e_2^2e_1(u)-(q+q^{-1})e_2e_1e_2(u)+e_1e_2^2(u)=0,\end{eqnarray}
\begin{eqnarray}\label{r14}f_1^2f_2(u)-(q+q^{-1})f_1f_2f_1(u)+f_2f_1^2(u)=0,\end{eqnarray}
\begin{eqnarray}\label{e12}f_2^2f_1(u)-(q+q^{-1})f_2f_1f_2(u)+f_1f_2^2(u)=0,\end{eqnarray}
where $u\in \{x_1, x_2\}$. Here, since the actions of $A$ and $B$ in $U_q(sl(3))$ can be interchanged, we only need to check 21 kinds of cases, i.e., whether $(Ai)$ is compatible with $(Bj)$ for any $1\leq i\leq j\leq 6$.
By some computations, we can find that $(A1)$ is compatible with $(B1)$, $(A2)$ is compatible with $(B5)$ only when $s_2=0$, $t_2=0$, $(A2)$ is compatible with $(B6)$ only when $u_2=0$, $v_2=0$, $(A5)$ is compatible with $(B6)$ only when $s_1=t_1=u_2=v_2=0$ and any other two are not compatible. Here, for example, we only check whether $(A5)$ and $(B6)$ are compatible. First, we need to check $k_1e_2(x_1)=q^{-1}e_2k_1(x_1)$. In this case, we obtain $k_1e_2(x_1)=-q^{-2}d_2^{-1}x_1x_2+q^{-6}u_2x_1^3$ and $q^{-1}e_2k_1(x_1)=-q^{-2}d_2^{-1}x_1x_2+q^{-3}u_2x_1^3$. Therefore, if $k_1e_2(x_1)=q^{-1}e_2k_1(x_1)$, we must have $u_2=0$. Similarly, in checking $k_1e_2(x_2)=q^{-1}e_2k_1(x_2)$ and $k_2f_1(u)=qf_1k_2(u)$ for $u\in \{x_1, x_2\}$, we obtain $v_2=0$, $s_1=0$ and $t_1=0$. Then, other relations in (\ref{e11})-(\ref{e12}) are easily checked when $u_2=0$, $v_2=0$, $s_1=0$ and $t_1=0$. In this case, all actions are isomorphic to that with $a_1$=1, $d_2=1$. The desired isomorphism is given by $\varphi:x_1\rightarrow a_1 x_1, x_2\rightarrow d_2 x_2$.
Other cases can be checked similarly. It should be pointed out that since all the automorphisms of $A_q(2)$ commute with the actions of $k_1$ and $k_2$, the first 16 kinds of actions of $U_q(sl(3))$ in the first case are pairwise nonisomorphic.
Thus, the classical limits in the table above are obtained.

Since every automorphism of $A_q(2)$ commutes with the actions of $k_1$ and $k_2$ and the actions of $k_1$ and $k_2$ in the four kinds of actions are different, there are no isomorphisms between these four kinds of module-algebra structures.
\end{proof}

Denote the actions of $U_q(sl(2))$ on $A_q(2)$ in Case (1), Case (2), Case (5) with $t_i=s_i=0$ and Case (6) with $u_i=v_i=0$ by $\ast1$, $\ast2$, $\ast3$ and $\ast4$ respectively. If $\ast s$ and $\ast t$ are compatible, in other words, they determine a
$U_q(sl(3))$-module algebra structure on $A_q(2)$, we use an edge connecting $\ast s$ and $\ast t$, since $k_1$, $e_1$, $f_1$ and
$k_2$, $e_2$, $f_2$ are symmetric in $U_q(sl(3))$. Therefore,
we can use the following diagrams to express all the module algebra structures of $U_q(sl(3))$ on $A_q(2)$ in Theorem \ref{T8}:
\begin{eqnarray}\label{w1}\xymatrix{{\ast 1} \ar@{-}[r]&{\ast 1}  },~~~~\xymatrix{
  {\ast2} \ar@{-}[rr] \ar@{-}[dr]
                &  &   {\ast3}\ar@{-}[dl]    \\
                & {\ast4}               }.\end{eqnarray}
Here, every two adjacent vertices corresponds to two classes of module-algebra structures of $U_q(sl(3))$ on $A_q(2)$. For example,
$\xymatrix{{\ast 2} \ar@{-}[r]&{\ast 3}  }$ corresponds to the following two kinds of module-algebra structures of $U_q(sl(3))$ on $A_q(2)$:
one is that the actions of $k_1$, $e_1$, $f_1$ are of type $\ast 2$ and the actions of $k_2$, $e_2$, $f_2$ are of type $\ast 3$; the other is that  the actions of $k_1$, $e_1$, $f_1$ are of type $\ast 3$ and the actions of $k_2$, $e_2$, $f_2$ are of type $\ast 2$.

Next, we study the module-algebra structures of $U_q(sl(m+1))$ on $A_q(2)$ for $m\geq 3$. The corresponding Dynkin diagram of $sl(m+1)$ with $m$
vertices is as follows:
$$\xymatrix{
{\circ} \ar@{-}[r]& {\circ} \ar@{-}[r]&{\cdots}\ar@{-}[r]& {\circ} \ar@{-}[r]&{\circ}}.$$
In $U_q(sl(m+1))$, every vertex corresponds to one Hopf subalgebra isomorphic to $U_q(sl(2))$ and two adjacent vertices correspond to one Hopf subalgebra isomorphic to $U_q(sl(3))$. Therefore, for studying the module-algebra structures of $U_q(sl(m+1))$ on $A_q(2)$, we need to endow every vertex one kinds of actions of $U_q(sl(2))$ on $A_q(2)$. Moreover, there are some important observations:
\begin{itemize}
\item[1.] Since two adjacent vertices in the Dynkin diagram correspond to one Hopf subalgebra isomorphic to $U_q(sl(3))$, by Theorem \ref{T8},
the action of $U_q(sl(2))$ on $A_q(2)$ (on every vertex) should be of the following four kinds of possibilities: $\ast1$, $\ast2$,
$\ast3$ and $\ast4$. Moreover, every two adjacent vertices should be of the types in (\ref{w1}).

\item[2.] Except $\ast1$, no other type of actions of $U_q(sl(2))$ on $A_q(2)$ can be endowed with two different vertices simultaneously, since the relations (\ref{w2}), acting on $x_1$, $x_2$ producing zero, can not be satisfied.
\item[3.]If every vertex in the Dynkin diagram of $sl(m+1)$ is endowed an action of $U_q(sl(2))$ on $A_q(2)$ which is not Case $\ast1$, any two vertices which are not adjacent can not be endowed with the types which are adjacent in (\ref{w1}).
\end{itemize}

By the above discussion, we can obtain the following theorem immediately.
\begin{theorem}
For any $m\geq 3$, all module-algebra structures of $U_q(sl(m+1))$ on $A_q(2)$ are as follows:
\begin{eqnarray*}
&&k_i(x_1)=\pm x_1, \qquad k_i(x_2)=\pm x_2,\\
&&e_i(x_1)=e_i(x_2)=f_i(x_1)=f_i(x_2)=0,
\end{eqnarray*}
for any $i\in\{1,\cdots, m\}$.
\end{theorem}

\section{Classification of $U_q(sl(m+1))$-symmetries on $A_q(3)$}
In this section we study the module-algebra structures of
$U_q(sl(m+1))$ on $A_q(3)$. As in Section 3, we first study the
actions of $U_q(sl(2))$ on $A_q(3)$.

Let us assume that $U_q(sl(2))$ is generated by $k$, $e$, $f$. By the definition of the module algebra, it is easy to see that any action of $U_q(sl(2))$ on $A_q(3)$ is determined by such $3\times 3$ matrix with entries from $A_q(3)$:
\begin{eqnarray}
M\stackrel{def}{=}\left[\begin{array}{ccc}
k(x_1)& k(x_2)& k(x_3)\\
e(x_1)& e(x_2)& e(x_3)\\
f(x_1)& f(x_2)& f(x_3)\end{array}\right],
\end{eqnarray}
which is called the action matrix (see \cite{dup/sin3}).
Given a $U_q(sl(2))$-module algebra structure on $A_q(3)$. Obviously, the action of $k$ determines an automorphism of $A_q(3)$. Therefore, by Proposition \ref{w3}, we can set
\begin{eqnarray}
M_k\stackrel{def}{=}\left[\begin{array}{ccc}
k(x_1)& k(x_2)& k(x_3)\end{array}\right]= \left[\begin{array}{ccc}
\alpha x_1& \beta x_2& \gamma x_3\end{array}\right],
\end{eqnarray}
where $\alpha$, $\beta$, $\gamma$ are non-zero complex numbers. So, every monomial $x_1^mx_2^nx_3^s\in A_q(3)$ is an eigenvector for $k$ and the
associated eigenvalue $\alpha^m\beta^n\gamma^s$ is called the \emph{weight} of this monomial, which will be written as $wt(x_1^mx_2^nx_3^s)=\alpha^m\beta^n\gamma^s$.

Let
\begin{eqnarray}
M_{ef}\stackrel{def}{=}\left[\begin{array}{ccc}
e(x_1)& e(x_2)& e(x_3)\\
f(x_1)& f(x_2)& f(x_3) \end{array}\right],
\end{eqnarray}
\begin{eqnarray*}
wt(M)\stackrel{def}{=}\left[\begin{array}{ccc}
wt(k(x_1))& wt(k(x_2))& wt(k(x_3))\\
wt(e(x_1))& wt(e(x_2))& wt(e(x_3))\\
wt(f(x_1))& wt(f(x_2))& wt(f(x_3)) \end{array}\right]
\bowtie \left[\begin{array}{ccc}
\alpha & \beta & \gamma\\
q^2\alpha & q^2\beta & q^2\gamma\\
q^{-2}\alpha & q^{-2}\beta & q^{-2}\gamma \end{array}\right],
\end{eqnarray*}
where the relation $A=(a_{ij}) \bowtie B=(b_{ij})$ means that if for every pair of indices $i$, $j$ such that both $a_{ij}$ and $b_{ij}$ are nonzero, one has $a_{ij}=b_{ij}$.

In the following, we denote the $i$-th homogeneous component of $M$, whose elements are just the $i$-th homogeneous components of the corresponding entries of $M$,  by $(M)_i$.

Set \begin{eqnarray*}
(M)_0=\left[\begin{array}{ccc}
0& 0 & 0\\
a_0& b_0& c_0\\
d_0& e_0& f_0 \end{array}\right]_0,
\end{eqnarray*}
where $a_0$, $b_0$, $c_0$, $d_0$, $e_0$, $f_0\in \mathbb{C}$.
Then, we obtain
\begin{eqnarray}\label{2}
wt((M)_0)
\bowtie \left[\begin{array}{ccc}
0& 0& 0\\
q^2\alpha& q^2\beta & q^2\gamma\\
q^{-2}\alpha & q^{-2}\beta & q^{-2}\gamma \end{array}\right]_0
\bowtie \left[\begin{array}{ccc}
0& 0& 0\\
1& 1& 1\\
1& 1& 1 \end{array}\right]_0 .
\end{eqnarray}
An application of $e$ and $f$ to Equation (\ref{e61}) gives the following six equalities:
\begin{eqnarray}\label{4}
x_2e(x_1)-q\beta e(x_1)x_2=qx_1e(x_2)-\alpha e(x_2)x_1,\end{eqnarray}
\begin{eqnarray}\label{e1}f(x_1)x_2-q^{-1}\beta^{-1}x_2f(x_1)=q^{-1}f(x_2)x_1-\alpha^{-1}x_1f(x_2),\end{eqnarray}
\begin{eqnarray}\label{e2}x_3e(x_1)-q\gamma e(x_1)x_3=qx_1e(x_3)-\alpha e(x_3)x_1,\end{eqnarray}
\begin{eqnarray}\label{e3}f(x_1)x_3-q^{-1}\gamma^{-1}x_3f(x_1)=q^{-1}f(x_3)x_1-\alpha^{-1}x_1f(x_3),\end{eqnarray}
\begin{eqnarray}\label{e4}x_3e(x_2)-q\gamma e(x_2)x_3=qx_2e(x_3)-\beta e(x_3)x_2,\end{eqnarray}
\begin{eqnarray}\label{5}f(x_2)x_3-q^{-1}\gamma^{-1}x_3f(x_2)=q^{-1}f(x_3)x_2-\beta^{-1}x_2f(x_3).
\end{eqnarray}

After projecting the above six relations to $(A_q(3))_1$, we get
\begin{eqnarray*}
&&a_0(1-q\beta)x_2=b_0(q-\alpha)x_1,~~~~d_0(q-\beta^{-1})x_2=e_0(1-q\alpha^{-1})x_1,\\
&&a_0(1-q\gamma)x_3=c_0(q-\alpha)x_1,~~~~d_0(q-\gamma^{-1})x_3=f_0(1-q\alpha^{-1})x_1,\\
&&b_0(1-q\gamma)x_3=c_0(q-\beta)x_2,~~~~e_0(q-\gamma^{-1})x_3=f_0(1-q\beta^{-1})x_2.
\end{eqnarray*}
Therefore, we obtain that
\begin{eqnarray*}
&&a_0(1-q\beta)=b_0(q-\alpha)=d_0(q-\beta^{-1})=e_0(1-q\alpha^{-1})\\
&&=a_0(1-q\gamma)=c_0(q-\alpha)=d_0(q-\gamma^{-1})=f_0(1-q\alpha^{-1})\\
&&=b_0(1-q\gamma)=c_0(q-\beta)=e_0(q-\gamma^{-1})=f_0(1-q\beta^{-1})=0.
\end{eqnarray*}
Then, we have
\begin{eqnarray}
&&a_0\neq 0 \Rightarrow \beta=q^{-1}, \gamma=q^{-1},\qquad b_0\neq 0 \Rightarrow \alpha=q, \gamma=q^{-1},\\
&&c_0\neq 0 \Rightarrow \alpha=q, \beta=q,\qquad d_0\neq 0 \Rightarrow \beta=q^{-1}, \gamma=q^{-1},\\
&&e_0\neq 0 \Rightarrow \alpha=q, \gamma=q^{-1},\qquad f_0\neq 0 \Rightarrow \alpha=q, \beta=q.
\end{eqnarray}
By (\ref{2}) and using the above six equalities, we have only seven possibilities as follows:
\begin{eqnarray}
&&\left[\begin{array}{ccc}
\star& 0& 0\\
0 & 0 & 0 \end{array}\right]_0\Rightarrow \alpha=q^{-2}, \beta=q^{-1}, \gamma=q^{-1},\\
&&\left[\begin{array}{ccc}
0& \star&0\\
0  & 0 & 0 \end{array}\right]_0\Rightarrow \alpha=q, \beta=q^{-2}, \gamma=q^{-1},\\
&&\left[\begin{array}{ccc}
0& 0& \star\\
0&  0 & 0 \end{array}\right]_0\Rightarrow \alpha=q, \beta=q, \gamma=q^{-2},\\
&&\left[\begin{array}{ccc}
0& 0&0\\
\star &  0 & 0 \end{array}\right]_0\Rightarrow \alpha=q^{2}, \beta=q^{-1}, \gamma=q^{-1},\\
&&\left[\begin{array}{ccc}
0& 0&0\\
0  &  \star &0 \end{array}\right]_0\Rightarrow \alpha=q, \beta=q^{2}, \gamma=q^{-1},\\
&&\left[\begin{array}{ccc}
0&0&0\\
0  & 0 &\star \end{array}\right]_0\Rightarrow \alpha=q, \beta=q, \gamma=q^{2},\\
&&\left[\begin{array}{ccc}
0& 0& 0\\
0 & 0 &0 \end{array}\right]_0.
\end{eqnarray}
Here, in the above matrices, $\star$ just means the entry in the corresponding position is nonzero.

For the 1-st homogeneous component, since $wt(e(x_1))=q^2wt(x_1)\neq wt(x_1)$, we have $(e(x_1))_1=a_1x_2+b_1 x_3,$ for some $a_1$, $b_1\in \mathbb{C}$. Therefore, we set
\begin{eqnarray}
(M_{ef})_1=\left[\begin{array}{ccc}
a_1x_2+b_1x_3& a_2x_1+b_2x_3&a_3x_1+b_3x_2\\
c_1x_2+d_1x_3&  c_2x_1+d_2x_3&c_3x_1+d_3x_2 \end{array}\right]_1,
\end{eqnarray}
where $a_i$, $b_i$, $c_i\in \mathbb{C}$ for $i\in\{1, 2, 3\}$.

After projecting (\ref{4})-(\ref{5}) to $(A_q(3))_2$, we can obtain
\begin{eqnarray*}
a_1(1-q\beta)x_2^2=b_1(1-q^2\beta)x_1x_2=a_2(q-\alpha)x_1^2=b_2(q-q\alpha)x_1x_3,\end{eqnarray*}
\begin{eqnarray*}c_1(1-q^{-1}\beta^{-1})x_2^2=d_1(q-q^{-1}\beta^{-1})x_2x_3=d_2(1-\alpha^{-1})x_1x_3=c_2(q^{-1}-\alpha^{-1})x_1^2,\end{eqnarray*}
\begin{eqnarray*}a_1(q-q\gamma)x_2x_3=b_1(1-q\gamma)x_3^2=a_3(q-\alpha)x_1^2=b_3(q-q\alpha)x_1x_2,\end{eqnarray*}
\begin{eqnarray*}c_1(1-\gamma^{-1})x_2x_3=d_1(1-q^{-1}\gamma^{-1})x_3^2=d_3(1-\alpha^{-1})x_1x_2=c_3(q^{-1}-\alpha^{-1})x_1^2,\end{eqnarray*}
\begin{eqnarray*}a_2(q-q\gamma)x_1x_3=b_2(1-q\gamma)x_3^2=a_3(q^2-\beta)x_1x_2=b_3(q-\beta)x_2^2,\end{eqnarray*}
\begin{eqnarray*}c_2(1-\gamma^{-1})x_1x_3=d_2(1-q^{-1}\gamma^{-1})x_3^2=c_3(q^{-1}-q\beta^{-1})x_1x_2=d_3(q^{-1}-\beta^{-1})x_2^2.
\end{eqnarray*}
Therefore, we have
\begin{eqnarray*}
a_1\neq 0\Rightarrow b_1=0, \beta=q^{-1}, \gamma=1, \qquad
b_1\neq 0\Rightarrow a_1=0, \beta=q^{-2}, \gamma=q^{-1},\end{eqnarray*}
\begin{eqnarray*}
a_2\neq 0\Rightarrow b_2=0, \alpha=q, \gamma=1,\qquad b_2\neq 0\Rightarrow a_2=0, \alpha=1, \gamma=q^{-1},\end{eqnarray*}
\begin{eqnarray*}a_3\neq 0\Rightarrow b_3=0, \alpha=q, \beta=q^2,\qquad b_3\neq 0\Rightarrow a_3=0, \alpha=1, \beta=q,\end{eqnarray*}
\begin{eqnarray*}c_1\neq 0 \Rightarrow d_1=0, \beta=q^{-1}, \gamma=1,\qquad d_1\neq 0 \Rightarrow c_1=0, \beta=q^{-2}, \gamma=q^{-1},\end{eqnarray*}
\begin{eqnarray*}c_2\neq 0 \Rightarrow d_2=0, \alpha=q, \gamma=1,\qquad d_2\neq 0\Rightarrow c_2=0, \alpha=1, \gamma=q^{-1},\end{eqnarray*}
\begin{eqnarray*}c_3\neq 0 \Rightarrow d_3=0, \alpha=q, \beta=q^2,\qquad d_3\neq 0 \Rightarrow c_3=0, \alpha=1, \beta=q.
\end{eqnarray*}

Since
$
wt((M_{ef})_1)=\left[\begin{array}{ccc}
q^2\alpha & q^2\beta & q^2\gamma\\
q^{-2}\alpha &q^{-2}\beta & q^{-2}\gamma\end{array}\right]_1$, by the above discussion, for the 1-st homogeneous component, we only have the following possibilities 
(here, $\star_i$ in the position of $e(u)$ means that the $i$-th monomial of $e(u)$ following the $x_1$, $x_2$, $x_3$ order is nonzero, where
$u\in\{x_1, x_2, x_3\}$):
\begin{eqnarray*}
&&\left[\begin{array}{ccc}
\star_1& 0& 0\\
0  & 0 & 0 \end{array}\right]_1\Rightarrow \alpha=q^{-3}, \beta=q^{-1}, \gamma=1,\\
&&\left[\begin{array}{ccc}
\star_2& 0& 0\\
0  &  0 &0 \end{array}\right]_1\Rightarrow \alpha=q^{-3}, \beta=q^{-2}, \gamma=q^{-1},\\
&&\left[\begin{array}{ccc}
0&\star_1&0\\
0 & 0 &0 \end{array}\right]_1\Rightarrow \alpha=q, \beta=q^{-1}, \gamma=1,\\
&&\left[\begin{array}{ccc}
0&\star_2&0\\
0 & 0 &0 \end{array}\right]_1\Rightarrow \alpha=1, \beta=q^{-3}, \gamma=q^{-1},\\
&&\left[\begin{array}{ccc}
0&0 &\star_1\\
0 &  0 &0 \end{array}\right]_1\Rightarrow \alpha=q, \beta=q^{2}, \gamma=q^{-1},\\
&&\left[\begin{array}{ccc}
0& 0&\star_2\\
0  &0 &0 \end{array}\right]_1\Rightarrow \alpha=1, \beta=q, \gamma=q^{-1},\\
&&\left[\begin{array}{ccc}
0 & 0 &0\\
\star_1 & 0 & 0 \end{array}\right]_1\Rightarrow \alpha=q, \beta=q^{-1}, \gamma=1,\\
&&\left[\begin{array}{ccc}
0& 0& 0\\
\star_2 & 0 &0 \end{array}\right]_1\Rightarrow \alpha=q, \beta=q^{-2}, \gamma=q^{-1},\\
&&\left[\begin{array}{ccc}
0&0&0\\
0 & \star_1 &0 \end{array}\right]_1\Rightarrow \alpha=q, \beta=q^{3}, \gamma=1,\\
&&\left[\begin{array}{ccc}
0& 0&0\\
0& \star_2 &0 \end{array}\right]_1\Rightarrow \alpha=1, \beta=q, \gamma=q^{-1},\\
&&\left[\begin{array}{ccc}
0&0& 0\\
0 & 0 & \star_1 \end{array}\right]_1\Rightarrow \alpha=q, \beta=q^2, \gamma=q^{3},\\
&&\left[\begin{array}{ccc}
0&0&0\\
0 & 0 & \star_2 \end{array}\right]_1\Rightarrow \alpha=1, \beta=q, \gamma=q^{3},\\
&&\left[\begin{array}{ccc}
0& \star_1& 0\\
\star_1 & 0&0 \end{array}\right]_1\Rightarrow \alpha=q, \beta=q^{-1}, \gamma=1,\\
&&\left[\begin{array}{ccc}
0& 0&\star_2\\
0&\star_2& 0 \end{array}\right]_1\Rightarrow \alpha=1, \beta=q, \gamma=q^{-1},\\
&&\left[\begin{array}{ccc}
0 &0 &0\\
0  & 0 & 0 \end{array}\right]_1.\end{eqnarray*}

Obviously, if both the 0-th homogeneous component and the 1-th homogeneous component of $M_{ef}$ are nonzero, there are only two possibilities
\begin{eqnarray}\label{7}\left(\left[\begin{array}{ccc}
0& \star& 0\\
0&  0 & 0 \end{array}\right]_0, \left[\begin{array}{ccc}
0&0&0\\
\star_2 & 0 &0 \end{array}\right]_1\right) \Rightarrow \alpha=q, \beta=q^{-2}, \gamma=q^{-1},\end{eqnarray}
\begin{eqnarray}\label{8}\left(\left[\begin{array}{ccc}
0& 0& 0\\
0 & \star & 0 \end{array}\right]_0,\left[\begin{array}{ccc}
0& 0& \star_1\\
0 & 0 & 0 \end{array}\right]_1\right)\Rightarrow \alpha=q, \beta=q^2, \gamma=q^{-1}.\end{eqnarray}

Moreover, there are no possibilities when the 0-th homogeneous component of $M_{ef}$ is 0 and the 1-th homogeneous component of $M_{ef}$ has only one nonzero position.
The reasons are the same as those in \cite{dup/sin3}.

So, we only need to consider the following 11 possibilities\\
 $\left(\left[\begin{array}{ccc}
0 & 0&0\\
0 & 0 & 0 \end{array}\right]_0, \left[\begin{array}{ccc}
0 &0&0\\
0 &  0 & 0 \end{array}\right]_1\right)$, $\left(\left[\begin{array}{ccc}
\star & 0& 0\\
0 &  0 &0 \end{array}\right]_0, \left[\begin{array}{ccc}
0 & 0& 0\\
0 & 0 & 0 \end{array}\right]_1\right)$,\\
$\left(\left[\begin{array}{ccc}
0 & 0 &0\\
0 & 0 & \star \end{array}\right]_0, \left[\begin{array}{ccc}
0 &0 &0\\
0 & 0 &0 \end{array}\right]_1\right)$, $\left(\left[\begin{array}{ccc}
0 & 0 & 0\\
0 &  0 & 0 \end{array}\right]_0, \left[\begin{array}{ccc}
0 & \star_1 & 0\\
\star_1 &  0 &0 \end{array}\right]_1\right)$,\\
$\left(\left[\begin{array}{ccc}
0 & 0 &0\\
0 &  0 & 0 \end{array}\right]_0, \left[\begin{array}{ccc}
0 &0 & \star_2\\
0 &  \star_2 &0 \end{array}\right]_1\right)$, $\left(\left[\begin{array}{ccc}
0 & \star &0\\
0 & 0 &0 \end{array}\right]_0, \left[\begin{array}{ccc}
0 &0& 0\\
0 & 0 & 0 \end{array}\right]_1\right)$,\\
 $\left(\left[\begin{array}{ccc}
0&\star& 0\\
0 & 0 & 0 \end{array}\right]_0, \left[\begin{array}{ccc}
0&0&0\\
\star_2 &0 & 0 \end{array}\right]_1\right)$,
$\left(\left[\begin{array}{ccc}
0 &0 & 0\\
0 & \star &0 \end{array}\right]_0, \left[\begin{array}{ccc}
0 & 0 & 0\\
0 & 0 &0 \end{array}\right]_1\right)$,\\ $\left(\left[\begin{array}{ccc}
0&0& 0\\
0 & \star &0 \end{array}\right]_0,\left[\begin{array}{ccc}
0& 0&\star_1\\
0 & 0 & 0 \end{array}\right]_1\right),$
$\left(\left[\begin{array}{ccc}
0 & 0 & \star\\
0 & 0 & 0 \end{array}\right]_0, \left[\begin{array}{ccc}
0 &0 & 0\\
0 & 0 &0 \end{array}\right]_1\right)$,\\
$\left(\left[\begin{array}{ccc}
0 &0 &0\\
\star &  0 &0 \end{array}\right]_0, \left[\begin{array}{ccc}
0 & 0 &0\\
0 & 0 &0 \end{array}\right]_1\right)$.

For convenience, we denote these 11 kinds of cases in the above order by $(\ast_1)$, $\cdots$, $(\ast_{11})$ respectively.
\begin{lemma}\label{r8}
For Case $(\ast_1)$, all $U_q(sl(2))$-module algebra structures on $A_q(3)$ are as follows
\begin{eqnarray}
k(x_1)=\pm x_1, ~~~k(x_2)=\pm x_2,~~~k(x_3)=\pm x_3,\\
e(x_1)=e(x_2)=e(x_3)=f(x_1)=f(x_2)=f(x_3)=0.\end{eqnarray}
\end{lemma}

\begin{proof}
The proof is similar to that in Theorem 4.2 in \cite{dup/sin3}.
\end{proof}
\begin{lemma}\label{10}
For Case $(\ast_2)$, all $U_q(sl(2))$-module algebra structures on $A_q(3)$ are
\begin{eqnarray}
k(x_1)=q^{-2}x_1, ~~k(x_2)=q^{-1}x_2,~~~k(x_3)=q^{-1}x_3,\\
e(x_1)=a_0,~~e(x_2)=0,~~~e(x_3)=0,\\
f(x_1)=-qa_0^{-1}x_1^2,\\
f(x_2)=-qa_0^{-1}x_1x_2+\xi_1x_2x_3^2+\xi_2x_2^3+\xi_3x_3^3,\\
f(x_3)=-qa_0^{-1}x_1x_3+\xi_4x_2x_3^2+(1+q+q^2)\xi_2x_2^2x_3-q^{-1}\xi_1x_3^3,
\end{eqnarray}
where $a_0\in \mathbb{C}\backslash\{0\}$, and $\xi_1$, $\xi_2$, $\xi_3$, $\xi_4\in \mathbb{C}$.
\end{lemma}
\begin{proof}
Since $wt(M_{ef})\bowtie\left[\begin{array}{ll}
1 \qquad q\qquad q\\
q^{-4} \quad  q^{-3} \quad q^{-3} \end{array}\right]$ and $\alpha=q^{-2}$, $\beta=q^{-1}$, $\gamma=q^{-1}$, we must have
$e(x_1)=a_0$, $e(x_2)=0$, $e(x_3)=0$. With the same reason, $f(x_1)$, $f(x_2)$, $f(x_3)$ must be of the following forms:
$$f(x_1)=u_1x_1^2+u_2x_1x_2^2+u_3x_1x_3^2+u_4x_1x_2x_3+u_5x_2x_3^3+u_6x_2^2x_3^2+u_7x_2^3x_3+u_8x_2^4+u_9x_3^4,$$
$$f(x_2)=v_1x_1x_2+v_2x_1x_3+v_3x_2x_3^2+v_4x_2^2x_3+v_5x_2^3+v_6x_3^3,$$
$$f(x_3)=w_1x_1x_2+w_2x_1x_3+w_3x_2x_3^2+w_4x_2^2x_3+w_5x_2^3+w_6x_3^3,$$
where these coefficients are in $\mathbb{C}$. Then, we consider (\ref{4})-(\ref{5}). Taking $e(x_1)$, $e(x_2)$, $e(x_3)$,
$f(x_1)$, $f(x_2)$, $f(x_3)$ into the six equalities, by comparing the coefficients, we obtain that
$$f(x_1)=u_1x_1^2,$$
$$f(x_2)=u_1x_1x_2+v_3x_2x_3^2+v_5x_2^3+v_6x_3^3,$$
$$f(x_3)=u_1x_1x_3+w_3x_2x_3^2+(1+q+q^2)v_5x_2^2x_3-q^{-1}v_3x_3^3.$$
Using $ef(u)-fe(u)=\frac{k-k^{-1}}{q-q^{-1}}(u),$ for any $u\in \{x_1, x_2, x_3\}$, we get
$u_1=-qa_0^{-1}$.
So we proved the lemma.
\end{proof}
\begin{lemma}\label{S1}
For Case $(\ast_3)$, all $U_q(sl(2))$-module algebra structures on $A_q(3)$ are as follows
\begin{eqnarray}
k(x_1)=qx_1, ~~k(x_2)=qx_2,~~~k(x_3)=q^{2}x_3,\\
e(x_1)=-qf_0^{-1}x_1x_3+\mu_1x_1^2x_2-q\mu_2x_1^3+(1+q+q^2)\mu_3x_1x_2^2,\\
e(x_2)=-qf_0^{-1}x_2x_3+\mu_2x_1^2x_2+\mu_3x_2^3+\mu_4x_1^3,\\
e(x_3)=-qf_0^{-1}x_3^2,\\
f(x_1)=0,~~f(x_2)=0,~~~f(x_3)=f_0,
\end{eqnarray}
where $f_0\in \mathbb{C}\backslash\{0\}$, and $\mu_1$, $\mu_2$, $\mu_3$ $\mu_4\in \mathbb{C}$.
\end{lemma}
\begin{proof}
The proof is similar to that in Lemma \ref{10}.
\end{proof}

\begin{lemma}\label{t1}
For Case $(\ast_4)$, to satisfy (\ref{4})-(\ref{5}), the actions of $k$, $e$, $f$ must be of the following form
\normalsize{
\begin{eqnarray*}
&&k(x_1)=qx_1,~~k(x_2)=q^{-1}x_2,~~k(x_3)=x_3,\\
&&e(x_1)=\sum_{\substack{m\geq 0,p\geq 0\\p\neq m+3}}a_{m,p}x_1^{m+3}x_2^mx_3^p+\sum_{m\geq 0}d_mx_1^{m+3}x_2^mx_3^{m+3},\\
&&e(x_2)=a_2x_1+\sum_{\substack{m\geq 0,p\geq 0\\p\neq m+3}}b_{m,p}x_1^{m+2}x_2^{m+1}x_3^p,\\
&&e(x_3)=\sum_{\substack{m\geq 0,p\geq 0\\p\neq m+3}}c_{m,p}x_1^{m+2}x_2^mx_3^{p+1}+\sum_{m\geq 0}e_mx_1^{m+2}x_2^mx_3^{m+4},\\
&&f(x_1)=c_1x_2+\sum_{\substack{m\geq 0, p\geq 0\\p\neq m+1}}d_{m,p}x_1^{m+1}x_2^{m+2}x_3^p+\sum_{m\geq 0}h_mx_1^{m+1}x_2^{m+2}x_3^{m+1},\\
&&f(x_2)=\sum_{\substack{m\geq 0, p\geq 0\\p\neq m+1}}e_{m,p}x_1^mx_2^{m+3}x_3^p,\\
&&f(x_3)=\sum_{\substack{m\geq 0, p\geq 0\\p\neq m+1}}g_{m,p}x_1^mx_2^{m+2}x_3^{p+1}+\sum_{m\geq 0}g_mx_1^mx_2^{m+2}x_3^{m+2},
\end{eqnarray*}
}
where $a_2, c_1\in \mathbb{C}\backslash\{0\}$, other coefficients are in $\mathbb{C}$ and $\frac{a_{m,p}}{b_{m,p}}=-\frac{(m+p+1)_q}{q^{p-1}(m+3-p)_q}$, $\frac{b_{m,p}}{c_{m,p}}=\frac{q^{p-1}(m+3-p)_q}{(2m+2)_q}$, $\frac{d_m}{e_m}=-\frac{(2m+4)_q}{(2m+2)_q}$, $\frac{d_{m,p}}{e_{m,p}}=-\frac{(m+p+3)_q}{q^{p+1}(m+1-p)_q}$, $\frac{d_{m,p}}{g_{m,p}}=-\frac{(m+p+3)_q}{q(2m+2)_q}$, $\frac{h_m}{g_m}=-\frac{(2m+4)_q}{q(2m+2)_q}$.

Especially, there are the following $U_q(sl(2))$-module algebra structures on $A_q(3)$
\begin{eqnarray}
\label{r1}k(x_1)=qx_1,~~k(x_2)=q^{-1}x_2,~~k(x_3)=x_3,\\
\label{r2}e(x_1)=0,~~e(x_2)=a_2x_1,~~e(x_3)=0,\\
\label{r3}f(x_1)=a_2^{-1}x_2,~~f(x_2)=0,~~f(x_3)=0,
\end{eqnarray}
where $a_2\in \mathbb{C}\backslash\{0\}$.
\end{lemma}
\begin{proof}
In this case, we get $\alpha=q$, $\beta=q^{-1}$, $\gamma=1$. Therefore, $wt(M_{ef})
\bowtie \left[\begin{array}{ccc}
q^3& q & q^2\\
q^{-1} &q^{-3} &q^{-2}\end{array}\right]$. Since $wt(e(x_1))=q^3$, $wt(e(x_2))=q$ and $wt(e(x_3))=q^2$, using the equalities (\ref{4}), (\ref{e2}) and (\ref{e4}) and by some computations, we can obtain $e(x_1)$, $e(x_2)$ and $e(x_3)$  as the forms in the lemma. Similarly, we also can determine the forms of $f(x_1)$, $f(x_2)$ and $f(x_3)$ in the lemma.

Moreover, it is easy to check that (\ref{r1})-(\ref{r3}) determine the module-algebra structures of $U_q(sl(2))$ on $A_q(3)$.
\end{proof}
\begin{lemma}\label{tt1}
For Case $(\ast_5)$, to satisfy (\ref{4})-(\ref{5}), the actions of $U_q(sl(2))$ on $A_q(3)$ must be of the following form
\normalsize{
\begin{eqnarray*}
&&k(x_1)=x_1,~~k(x_2)=qx_2,~~k(x_3)=q^{-1}x_3,\\
&&e(x_1)=\sum_{\substack{m\geq 0, p\geq 0\\p\neq m+1}}\widetilde{a_{m,p}}x_1^{p+1}x_2^{2+m}x_3^m+\sum_{m\geq 0}\widetilde{a_m}x_1^{m+2}x_2^{m+2}x_3^m,\\
&&e(x_2)=\sum_{\substack{m\geq 0, p\geq 0\\p\neq m+1}}\widetilde{b_{m,p}}x_1^px_2^{3+m}x_3^m,\\
&&e(x_3)=b_3x_2+\sum_{\substack{m\geq 0, p\geq 0\\p\neq m+1}}\widetilde{c_{m,p}}x_1^px_2^{m+2}x_3^{m+1}+\sum_{m\geq 0}\widetilde{c_m}x_1^{m+1}x_2^{m+2}x_3^{m+1},\\
&&f(x_1)=\sum_{\substack{m\geq 0, p\geq 0\\p\neq m+3}}\widetilde{d_{m,p}}x_1^{p+1}x_2^mx_3^{m+2}+\sum_{m\geq 0}\widetilde{d_m}x_1^{m+4}x_2^mx_3^{m+2},\\
&&f(x_2)=d_2x_3+\sum_{\substack{m\geq 0, p\geq 0\\p\neq m+3}}\widetilde{e_{m,p}}x_1^px_2^{m+1}x_3^{m+2},\\
&& f(x_3)=\sum_{\substack{m\geq 0, p\geq 0\\p\neq m+3}}\widetilde{g_{m,p}}x_1^px_2^mx_3^{m+3}+\sum_{m \geq 0}\widetilde{g_m}x_1^{m+3}x_2^mx_3^{m+3},
\end{eqnarray*}}
where $b_3$, $d_2\in \mathbb{C}\backslash\{0\}$, other coefficients are in $\mathbb{C}$ and $\frac{\widetilde{a_{m,p}}}{\widetilde{b_{m,p}}}=\frac{(2m+2)_q}{q^p(m-p+1)_q}$, $\frac{\widetilde{a_{m,p}}}{\widetilde{c_{m,p}}}=-\frac{q(2m+2)_q}{(m+p+3)_q}$, $\frac{\widetilde{a_m}}{\widetilde{c_m}}=-\frac{q(2m+2)_q}{(2m+4)_q}$, $\frac{\widetilde{d_{m,p}}}{\widetilde{e_{m,p}}}=\frac{(2m+2)_q}{q^{p-1}(m-p+3)_q}$, $\frac{\widetilde{d_{m,p}}}{\widetilde{g_{m,p}}}=-\frac{(2m+2)_q}{(m+p+1)_q}$, $\frac{\widetilde{d_m}}{\widetilde{g_m}}=-\frac{(2m+2)_q}{(2m+4)_q}$.

There are the following $U_q(sl(2))$-module algebra structures on $A_q(3)$
\begin{eqnarray}
k(x_1)=x_1,~~k(x_2)=qx_2,~~k(x_3)=q^{-1}x_3,\\
e(x_1)=0,~~e(x_2)=0,~~e(x_3)=b_3x_2,\\
f(x_1)=0,~~f(x_2)=b_3^{-1}x_3,~~f(x_3)=0,
\end{eqnarray}
where $b_3\in \mathbb{C}\backslash\{0\}$.
\end{lemma}
\begin{proof}
The proof is similar to that in Lemma \ref{t1}.
\end{proof}

\begin{lemma}\label{t3}
For Case $(\ast_6)$ and Case $(\ast_7)$, to satisfy (\ref{4})-(\ref{5}), the actions of $k$, $e$ and $f$ on $A_q(3)$ are
\normalsize{
\begin{eqnarray*}
k(x_1)&=&qx_1,~~k(x_2)=q^{-2}x_2,~~k(x_3)=q^{-1}x_3,\\
e(x_1)&=&\sigma x_1^3+\sum_{n\geq 0}\sigma_nx_1^{2n+5}x_2^{n+1}+\sum_{p\geq 0}\widetilde{\sigma_p}x_1^{p+4}x_3^{p+1}\\
&&+\sum_{n\geq 0,p\geq 0}\sigma_{n,p}x_1^{2n+p+6}x_2^{n+1}x_3^{p+1},\\
e(x_2)&=&b_0+\rho x_1^2x_2+\sum_{n\geq 0}\rho_nx_1^{2n+4}x_2^{n+2}+\sum_{p\geq 0}\widehat{\rho_p}x_1^{p+3}x_2x_3^{p+1}\\
&&+\sum_{n\geq 0,p\geq 0}\rho_{n,p}x_1^{2n+p+5}x_2^{n+2}x_3^{p+1},\\
e(x_3)&=&\tau x_1^2x_3+\sum_{n\geq 0}\tau_nx_1^{2n+4}x_2^{n+1}x_3+\sum_{p\geq 0}\widetilde{\tau_p}x_1^{p+3}x_3^{p+2}\\
&&+\sum_{n\geq 0,p\geq 0}\tau_{n,p}x_1^{2n+p+5}x_2^{n+1}x_3^{p+2},\\
f(x_1)&=&d_1x_3+\sum_{p\geq 0}\lambda_px_1^{p+1}x_3^{p+2}+\sum_{n\geq 0}\widetilde{\lambda_n}x_1^{2n+1}x_2^{n+1}+\sum_{n\geq 0}\widehat{\lambda_n}x_1^{2n+2}x_2^{n+1}x_3\\
&&+\sum_{n\geq 0,p\geq 0}\lambda_{n,p}x_1^{2n+p+3}x_2^{n+1}x_3^{p+2},\\
f(x_2)&=&\sum_{n\geq 0}\widetilde{\nu_n}x_1^{2n}x_2^{n+2}+\sum_{n\geq 0}\widehat{\nu_n}x_1^{2n+1}x_2^{n+2}x_3+\sum_{n\geq 0,p\geq 0}\nu_{n,p}x_1^{2n+p+2}x_2^{n+2}x_3^{p+2},\\
f(x_3)&=&\sum_{p\geq 0}\omega_px_1^px_3^{p+3}+\sum_{n\geq 0}\widetilde{\omega_n}x_1^{2n}x_2^{n+1}x_3+\sum_{n\geq 0}\widehat{\omega_n}x_1^{2n+1}x_2^{n+1}x_3^2\\
&&+\sum_{n\geq 0, p\geq 0}\omega_{n,p}x_1^{2n+p+2}x_2^{n+1}x_3^{p+3},\end{eqnarray*}
}
where $b_0\in \mathbb{C}\backslash\{0\}$ and other coefficients are in the $\mathbb{C}$, and
$\frac{\sigma}{\rho}=-\frac{q^2}{(4)_q}$, $\frac{\sigma_n}{\rho_n}=-\frac{q^2(n+2)_q}{(2n+6)_q}$, $\frac{\widetilde{\sigma_p}}{\widetilde{\rho_p}}=-\frac{(p+2)_q}{q^{p-1}(4)_q}$, $\frac{\sigma_{n,p}}{\rho_{n,p}}=-\frac{(n+p+3)_q}{q^{p-1}(2n+6)_q}$,
$\frac{\sigma}{\tau}=-\frac{q}{(3)_q}$, $\frac{\sigma_n}{\tau_n}=-\frac{q(n+2)_q}{(3n+6)_q}$, $\frac{\widetilde{\sigma_p}}{\widetilde{\tau_p}}=-\frac{q(p+2)_q}{(p+4)_q}$, $\frac{\sigma_{n,p}}{\tau_{n,p}}=-\frac{q(n+p+3)_q}{(3n+p+7)_q}$, $\frac{\lambda_{p}}{\omega_p}=-\frac{(p+3)_q}{q(p+1)_q}$, $\frac{\widetilde{\lambda_n}}{\widetilde{\omega_n}}=-\frac{(n+2)_q}{q(3n+2)_q}$,
$\frac{\widehat{\lambda_n}}{\widehat{\omega_n}}=-\frac{(n+3)_q}{q(3n+3)_q}$, $\frac{\lambda_{n,p}}{\omega_{n,p}}=-\frac{(n+p+4)_q}{q(3n+p+4)_q}$,
$\frac{\widetilde{\lambda_n}}{\widetilde{\nu_n}}$ $=$ $-\frac{(n+2)_q}{q(2n+2)_q}$,
$\frac{\widehat{\lambda_n}}{\widehat{\nu_n}}=-\frac{(n+3)_q}{q^2(2n+2)_q}$, $\frac{\lambda_{n,p}}{\nu_{n,p}}=-\frac{(n+p+4)_q}{q^{p+3}(2n+2)_q}$.

Especially, there are the following $U_q(sl(2))$-module algebra structures on $A_q(3)$:
\begin{eqnarray}
\label{r4}k(x_1)=qx_1,~~k(x_2)=q^{-2}x_2,~~k(x_3)=q^{-1}x_3,\end{eqnarray}
\begin{eqnarray}\label{r5} e(x_1)=0,~~e(x_2)=b_0,~~e(x_3)=0,\end{eqnarray}
\begin{eqnarray}\label{r6}f(x_1)=d_1x_3+b_0^{-1}x_1x_2+\sum_{p= 0}^n\widehat{d_p}x_1^{p+1}x_3^{p+2},~~f(x_2)=-qb_0^{-1}x_2^2,\end{eqnarray}
\begin{eqnarray}\label{r7}f(x_3)=-qb_0^{-1}x_2x_3-\sum_{p=0}^n\frac{q(p+1)_q}{(p+3)_q}\widehat{d_p}x_1^{p}x_3^{p+3},
\end{eqnarray}
where $n\in \mathbf{N}$, $d_1$, $\widehat{d_p}\in \mathbb{C}$ for all $p$, $b_0\in \mathbb{C}\backslash \{0\}$.

\end{lemma}
\begin{proof}
In these two cases, we have the same values of $\alpha$, $\beta$ and $\gamma$, i.e., $\alpha=q$, $\beta=q^{-2}$, $\gamma=q^{-1}$. Therefore, $wt(M_{ef})
\bowtie \left[\begin{array}{ccc}
q^3&1&q\\
q^{-1} &q^{-4} &q^{-3}\end{array}\right]$.  Using the equalities (\ref{4})-(\ref{5}) and by some computations, we can obtain that $e(x_1)$, $e(x_2)$, $e(x_3)$, $f(x_1)$, $f(x_2)$, $f(x_3)$ are of the forms in this lemma.

Moreover, using (\ref{r4})-(\ref{r7}), it is easy to check that $ef(u)-fe(u)=\frac{k-k^{-1}}{q-q^{-1}}(u)$, where $u\in \{x_1, x_2, x_3\}$.
Therefore, they determine the module-algebra structures of $U_q(sl(2))$ on $A_q(3)$.
\end{proof}

\begin{lemma}\label{S2}
For Case $(\ast_8)$ and Case $(\ast_9)$, to satisfy (\ref{4})-(\ref{5}), the actions of $k$, $e$, $f$ are of the form
\normalsize{
\begin{eqnarray*}
k(x_1)&=&qx_1,~~k(x_2)=q^{2}x_2,~~k(x_3)=q^{-1}x_3,\\
e(x_1)&=&\sum_{p\geq 0}\alpha_px_1^{p+3}x_3^p+\sum_{m\geq 0}\widetilde{\alpha_m}x_1x_2^{m+1}x_3^{2m}+\sum_{m\geq 0}\widehat{\alpha_m}x_1^2x_2^{m+1}x_3^{2m+1}\\
&&+\sum_{p\geq 0, m\geq 0}\alpha_{m,p}x_1^{p+3}x_2^{m+1}x_3^{2m+p+2},\\
e(x_2)&=&\sum_{m\geq 0}\widetilde{\beta_m}x_2^{m+2}x_3^{2m}+\sum_{m\geq 0}\widehat{\beta_m}x_1x_2^{m+2}x_3^{2m+1}\\
&&+\sum_{p\geq 0, m\geq 0}\beta_{m,p}x_1^{p+2}x_2^{m+2}x_3^{2m+p+2},\\
e(x_3)&=&a_3x_1+\sum_{p\geq 0}\gamma_px_1^{p+2}x_3^{p+1}+\sum_{m\geq 0}\widetilde{\gamma_m}x_2^{m+1}x_3^{2m+1}+\sum_{m\geq 0}\widehat{\gamma_m}x_1x_2^{m+1}x_3^{2m+2}\\
&&+\sum_{p\geq 0, m\geq 0}\gamma_{m,p}x_1^{p+2}x_2^{m+1}x_3^{2m+p+3},\\
f(x_1)&=&\varepsilon x_1x_3^2+\sum_{p\geq 0}\varepsilon_p x_1^{p+2}x_3^{p+3}+\sum_{m\geq 0}\widetilde{\varepsilon_m}x_1x_2^{m+1}x_3^{2m+4}\\
&&+\sum_{m\geq 0, p\geq 0}\varepsilon_{m,p}x_1^{p+2}x_2^{m+1}x_3^{2m+p+5},\\
f(x_2)&=&e_0+\theta x_2x_3^2+\sum_{p\geq 0}\theta_p x_1^{p+1}x_2x_3^{p+3}+\sum_{m\geq 0}\widetilde{\theta_m}x_2^{m+2}x_3^{2m+4}\\
&&+\sum_{m\geq 0, p\geq 0}\theta_{m,p}x_1^{p+1}x_2^{m+2}x_3^{2m+p+5},\\
f(x_3)&=&\eta x_3^3+\sum_{p\geq 0}\eta_p x_1^{p+1}x_3^{p+4}+\sum_{m\geq 0}\widetilde{\eta_m}x_2^{m+1}x_3^{2m+5}\\
&&+\sum_{m\geq 0, p\geq 0}\eta_{m,p}x_1^{p+1}x_2^{m+1}x_3^{2m+p+6},
\end{eqnarray*}
}
where $e_0 \in \mathbb{C}\backslash\{0\}$ and other coefficients are in the $\mathbb{C}$, and
$\frac{\alpha_p}{\gamma_p}=-\frac{q(p+1)_q}{(p+3)_q}$, $\frac{\widetilde{\alpha_m}}{\widetilde{\beta_m}}=\frac{(3m+2)_q}{(2m+2)_q}$,
$\frac{\widehat{\alpha_m}}{\widehat{\beta_m}}=\frac{(3m+3)_q}{q(2m+2)_q}$, $\frac{\alpha_{m,p}}{\beta_{m,p}}=\frac{(3m+p+4)_q}{q^{p+2}(2m+2)_q}$, $\frac{\widetilde{\alpha_m}}{\widetilde{\gamma_m}}=-\frac{q(3m+2)_q}{(m+2)_q}$,
$\frac{\widehat{\alpha_m}}{\widehat{\gamma_m}}=-\frac{q(3m+3)_q}{(m+3)_q}$,
$\frac{\alpha_{m,p}}{\gamma_{m,p}}=-\frac{q(3m+p+4)_q}{(p+m+4)_q}$,
$\frac{\varepsilon}{\theta}=\frac{q(3)_q}{(4)_q}$, $\frac{\varepsilon_p}{\theta_p}=\frac{(p+4)_q}{q^p(4)_q}$,
$\frac{\widetilde{\varepsilon_m}}{\widetilde{\theta_m}}=\frac{q(3m+6)_q}{(2m+6)_q}$,
$\frac{\varepsilon_{m,p}}{\theta_{m,p}}=\frac{(3m+p+7)_q}{q^p(2m+6)_q}$,
$\frac{\varepsilon}{\eta}=-q^{-1}(3)_q$,
 $\frac{\varepsilon_p}{\eta_p}=-\frac{(p+4)_q}{q(p+2)_q}$,
  $\frac{\widetilde{\varepsilon_m}}{\widetilde{\eta_p}}=-\frac{(3m+6)_q}{q(m+2)_q}$,
  $\frac{\varepsilon_{m,p}}{\eta_{m,p}}=-\frac{(3m+p+7)_q}{q(p+m+3)_q}$.

There are the following $U_q(sl(2))$-module algebra structures on $A_q(3)$
\begin{eqnarray}
k(x_1)=qx_1,~~k(x_2)=q^{2}x_2,~~k(x_3)=q^{-1}x_3,\end{eqnarray}
\begin{eqnarray}e(x_1)=-qe_0^{-1}x_1x_2-\sum_{p=0}^n\frac{q(p+1)_q}{(p+3)_q}\alpha_px_1^{p+3}x_3^p,~~e(x_2)=-qe_0^{-1}x_2^2,\end{eqnarray}
\begin{eqnarray} e(x_3)=a_3x_1+e_0^{-1}x_2x_3+\sum_{p=0}^n\alpha_px_1^{p+2}x_3^{p+1},\end{eqnarray}
\begin{eqnarray}f(x_1)=0,~~f(x_2)=e_0,~~f(x_3)=0,
\end{eqnarray}
where $n\in \mathbf{N}$, $a_3$, $\alpha_p\in \mathbb{C}$ for all $p$, $e_0\in \mathbb{C}\backslash \{0\}$.
\end{lemma}
\begin{proof}
The proof is similar to that in Lemma \ref{t3}.
\end{proof}

\begin{lemma}\label{t4}
For Case $(\ast_{10})$, to satisfy (\ref{4})-(\ref{5}), the actions of $k$, $e$, $f$ on $A_q(3)$ are
\normalsize{
\begin{eqnarray*}
&&k(x_1)=qx_1,~~k(x_2)=qx_2,~~k(x_3)=q^{-2}x_3,\\
&&e(x_1)=\sum_{\substack{n\geq 0,p\geq 0\\2+2p-n\geq 0\\
n\neq p+1}}r_{n,p}x_1^{3+2p-n}x_2^nx_3^p+\sum_{p\geq 0}r_px_1^{2+p}x_2^{p+1}x_3^p,\\
&&e(x_2)=\sum_{\substack{n\geq 0,p\geq 0\\2+2p-n\geq 0\\n\neq p+1}}s_{n,p}x_1^{2+2p-n}x_2^{n+1}x_3^p,\\
&&e(x_3)=c_0+\sum_{\substack{n\geq 0,p\geq 0\\2+2p-n\geq 0\\n\neq p+1}}t_{n,p}x_1^{2+2p-n}x_2^nx_3^{p+1}+\sum_{p\geq 0}t_px_1^{p+1}x_2^{p+1}x_3^{p+1},\\
&&f(x_1)=\sum_{\substack{n\geq 0,p\geq 0\\2p-n\geq 0\\
p\neq n+2}}u_{n,p}x_1^{2p-n+1}x_2^nx_3^{p+1}+\sum_{n\geq 0}u_nx_1^{n+5}x_2^{n}x_3^{n+3},\\
&&f(x_2)=\sum_{\substack{n\geq 0,p\geq 0\\2p-n\geq 0\\p\neq n+2}}v_{n,p}x_1^{2p-n}x_2^{n+1}x_3^{p+1},\\
&&f(x_3)=\sum_{\substack{n\geq 0,p\geq 0\\2p-n\geq 0\\p\neq n+2}}w_{n,p}x_1^{2p-n}x_2^nx_3^{p+2}+\sum_{n\geq 0}w_nx_1^{n+4}x_2^{n}x_3^{n+4},
\end{eqnarray*}}
where $c_0 \in \mathbb{C}\backslash\{0\}$ and other coefficients are in the $\mathbb{C}$, and
$\frac{r_{n,p}}{s_{n,p}}=\frac{(n+p+1)_q}{q^{p+1}(p+1-n)_q}$,
 $\frac{r_{n,p}}{t_{n,p}}=-\frac{q^2(n+p+1)_q}{(2p+4)_q}$, $\frac{r_p}{t_p}=-\frac{q^2(2p+2)_q}{(2p+4)_q}$,
 $\frac{u_{n,p}}{v_{n,p}}=-\frac{(n+p+2)_q}{q^{p+2}(p-2-n)_q}$,  $\frac{u_{n,p}}{w_{n,p}}=-\frac{(n+p+2)_q}{q(2p+2)_q}$,
  $\frac{u_{n}}{w_{n}}=-\frac{(2n+4)_q}{q(2n+6)_q}$.

Specifically, there are the following $U_q(sl(2))$-module algebra structures on $A_q(3)$
\begin{eqnarray}
k(x_1)=qx_1,~~k(x_2)=qx_2,~~k(x_3)=q^{-2}x_3,\\
e(x_1)=0,~~e(x_2)=0,~~e(x_3)=c_0,\\
f(x_1)=c_0^{-1}x_1x_3,~~f(x_2)=c_0^{-1}x_2x_3,~~f(x_3)=-qc_0^{-1}x_3^2,
\end{eqnarray}
where $c_0\in \mathbb{C}\backslash \{0\}$.

\end{lemma}
\begin{proof}
In this case, we have $\alpha=q$, $\beta=q$, $\gamma=q^{-2}$. Therefore, $wt(M_{ef})
\bowtie \left[\begin{array}{ccc}
q^3& q^3&1\\
q^{-1} &q^{-1} & q^{-4}\end{array}\right]$. Then, the proof is similar to those in the above lemmas.

\end{proof}

\begin{lemma}\label{r9}
For Case $(\ast_{11})$, to satisfy (\ref{4})-(\ref{5}), the actions of $k$, $e$ and $f$ on $A_q(3)$ are
\normalsize{
\begin{eqnarray*}
&&k(x_1)=q^2x_1,~~k(x_2)=q^{-1}x_2,~~k(x_3)=q^{-1}x_3,\\
&&e(x_1)=\sum_{\substack{m\geq 0,p\geq 0\\2m-p\geq 0\\
m\neq p+2}}\widetilde{r_{m,p}}x_1^{m+2}x_2^px_3^{2m-p}+\sum_{p\geq 0}\widetilde{r_p}x_1^{p+4}x_2^{p}x_3^{p+4},\\
&&e(x_2)=\sum_{\substack{m\geq 0,p\geq 0\\2m-p\geq 0\\m\neq p+2}}\widetilde{s_{m,p}}x_1^{m+1}x_2^{p+1}x_3^{2m-p},\\
&&e(x_3)=\sum_{\substack{m\geq 0,p\geq 0\\2m-p\geq 0\\m\neq p+2}}\widetilde{t_{n,p}}x_1^{m+1}x_2^px_3^{2m-p+1}+\sum_{p\geq 0}\widetilde{t_p}x_1^{p+3}x_2^{p}x_3^{p+5},\\
&&f(x_1)=d_0+\sum_{\substack{m\geq 0,p\geq 0\\2m+2-p\geq 0\\
p\neq m+1}}\widetilde{u_{m,p}}x_1^{m+1}x_2^px_3^{2m+2-p}+\sum_{m\geq 0}\widetilde{u_m}x_1^{m+1}x_2^{m+1}x_3^{m+1},\\
&&f(x_2)=\sum_{\substack{m\geq 0,p\geq 0\\2m+2-p\geq 0\\p\neq m+1}}\widetilde{v_{m,p}}x_1^{m}x_2^{p+1}x_3^{2m-p+2},\\
&&f(x_3)=\sum_{\substack{n\geq 0,p\geq 0\\2m+2-p\geq 0\\p\neq m+1}}\widetilde{w_{m,p}}x_1^{m}x_2^px_3^{2m-p+3}+\sum_{m\geq 0}\widetilde{w_m}x_1^{m}x_2^{m+1}x_3^{m+2},
\end{eqnarray*}}
where $d_0 \in \mathbb{C}\backslash\{0\}$ and other coefficients are in $\mathbb{C}$, and
$\frac{\widetilde{r_{m,p}}}{\widetilde{s_{m,p}}}=\frac{(2m+2)_q}{q^{m+1}(m-2-p)_q}$,
 $\frac{\widetilde{r_{m,p}}}{\widetilde{t_{m,p}}}=-\frac{q(2m+2)_q}{(m+p+2)_q}$, $\frac{r_p}{t_p}=-\frac{q(2p+6)_q}{(2p+4)_q}$,
 $\frac{\widetilde{u_{m,p}}}{\widetilde{v_{m,p}}}=\frac{(2m+4)_q}{q^{m+3}(m+1-p)_q}$,  $\frac{\widetilde{u_{m,p}}}{\widetilde{w_{m,p}}}=-\frac{(2m+4)_q}{q^2(m+p+1)_q}$,
  $\frac{\widetilde{u_{m}}}{\widetilde{w_{m}}}=-\frac{(2m+4)_q}{q^2(2m+2)_q}$.

There are the following $U_q(sl(2))$-module algebra structures on $A_q(3)$
\begin{eqnarray}
&&k(x_1)=q^2x_1,~~k(x_2)=q^{-1}x_2,~~k(x_3)=q^{-1}x_3,\\
&&e(x_1)=-qd_0^{-1}x_1^2,~~e(x_2)=d_0^{-1}x_1x_2,~~e(x_3)=d_0^{-1}x_1x_3,\\
&&f(x_1)=d_0,~~f(x_2)=0,~~f(x_3)=0,
\end{eqnarray}
where $d_0\in \mathbb{C}\backslash \{0\}$.
\end{lemma}
\begin{proof}
The proof is similar to that in Lemma \ref{t4}.
\end{proof}

Now, we begin to classify all module-algebra structures of $U_q(sl(3))$ on $A_q(3)$.

Denote the nine cases of the actions of $k_1$, $e_1$, $f_1$ (resp. $k_2$, $e_2$, $f_2$) in Lemma \ref{r8}-Lemma \ref{r9} by $(A_1)$, $\cdots$,
$(A9)$ (resp. $(B_1)$, $\cdots$,
$(B9)$). To determine all module-algebra structures of $U_q(sl(3))$ on $A_q(3)$, we have to find all the actions of $k_1$, $e_1$, $f_1$ and $k_2$, $e_2$, $f_2$ which are compatible, i.e., (\ref{e11})-(\ref{e12}) hold and $e_if_i(u)-f_ie_i(u)=\frac{k_i-k_i^{-1}}{q-q^{-1}}(u)$
holds for $u\in\{x_1, x_2, x_3\}$ and $i\in\{1,2\}$.

Because the actions of $k_1$, $e_1$, $f_1$ and $k_2$, $e_2$, $f_2$ in $U_q(sl(3))$ are symmetric, we only need to check 45 cases, i.e., whether
$(Ai)$ is compatible with $(Bj)$ for any $1\leq i\leq j\leq 9$. We use $(Ai)|(B_j)$ to denote that the actions of $k_1$, $e_1$, $f_1$ are those in $(Ai)$ and the actions of $k_2$, $e_2$, $f_2$ are those in $(Bj)$. Moreover, in Case $(Aj)|(Bj)$ ($j\geq 2$), since the actions of $e_i$, $f_i$ are not zero simultaneously for $i\in \{1,2\}$, (3.1) and (3.2) can not be satisfied simultaneously. Therefore, Cases $(Aj)|(Bj)$ ($j\geq 2$)
are excluded.

First, let us consider Case $(A2)|(B5)$. Since $k_2e_1(x_1)=k_2(a_0)=a_0$, $q^{-1}e_1k_2(x_1)$\\
$=$$q^{-1}a_0$ and $a_0\neq 0$,
$k_2e_1(x_1)=q^{-1}e_1k_2(x_1)$ does not hold. Therefore, $(A2)|(B5)$ should be excluded. For the same reason, we exclude
$(A2)|(B6)$, $(A2)|(B8)$, $(A2)|(B9)$, $(A3)|(B4)$, $(A3)|(B7)$, $(A3)|(B8)$, $(A3)|(B9)$, $(A4)|(B6)$, $(A4)|(B7)$,  $(A4)|(B8)$,  $(A4)|(B9)$,  $(A5)|(B7)$, $(A5)|(B8)$, $(A5)|(B9)$, $(A6)|(B7)$, $(A6)|(B8)$, $(A6)|(B9)$, $(A7)|(B8)$, $(A7)|(B9)$, $(A8)|(B9)$.

Second, we consider $(A1)|(B2)$. Since $k_1f_2(x_1)=-qa_0^{-1}x_1^2$ and $qf_2k_1(x_1)=\mp q^2a_0^{-1}x_1^2$, we have $k_1f_2(x_1)\neq qf_2k_1(x_1)$.
Thus, $(A_1)|(B2)$ should be excluded. Similarly, $(A1)|(Bi)$ should be excluded for $i\geq 3$.

Therefore, we only need to consider the following cases: $(A1)|(B1)$, $(A2)|(B3)$, $(A2)|(B4)$, $(A2)|(B7)$, $(A3)|(B5)$, $(A3)|(B6)$,
$(A4)|(B5)$, $(A4)|(B7)$, $(A5)|(B6)$.

\begin{lemma}\label{SS1}
For Case $(A1)|(B1)$, all module-algebra structures of $U_q(sl(3))$ on $A_q(3)$ are as follows
\begin{eqnarray*}
&&k_1(x_1)=\pm x_1, ~~~k_1(x_2)=\pm x_2,~~~k_1(x_3)=\pm x_3,\\
&&k_2(x_1)=\pm x_1, ~~~k_2(x_2)=\pm x_2,~~~k_2(x_3)=\pm x_3,\\
&&e_1(x_1)=e_1(x_2)=e_1(x_3)=f_1(x_1)=f_1(x_2)=f_1(x_3)=0,\\
&&e_2(x_1)=e_2(x_2)=e_2(x_3)=f_2(x_1)=f_2(x_2)=f_2(x_3)=0,
\end{eqnarray*}
which are pairwise non-isomorphic.
\end{lemma}
\begin{proof}
It can be seen that (\ref{e11})-(\ref{e12}) are satisfied for any $u\in\{x_1, x_2, x_3\}$ in this case. Therefore, they are module-algebra structures of $U_q(sl(3))$ on $A_q(3)$. Since all the automorphisms of $A_q(3)$ commute with the actions of $k_1$ and $k_2$, all these module-algebra structures are pairwise non-isomorphic.
\end{proof}

\begin{lemma}\label{r20}
For Case $(A2)|(B3)$, all $U_q(sl(3))$-module algebra structures on $A_q(3)$ are as follows:
\begin{eqnarray*}
&&k_1(x_1)=q^{-2}x_1, ~~k_1(x_2)=q^{-1} x_2,~~k_1(x_3)=q^{-1} x_3,\\
&&k_2(x_1)=qx_1, ~~k_2(x_2)=q x_2,~~k_2(x_3)=q^2 x_3,\\
&&e_1(x_1)=a_0,~~e_1(x_2)=0,~~e_1(x_3)=0,\\
&&e_2(x_1)=-qf_0^{-1}x_1x_3,~~e_2(x_2)=-qf_0^{-1}x_2x_3,~~e_2(x_3)=-qf_0^{-1}x_3^2,\\
&&f_1(x_1)=-qa_0^{-1}x_1^2,~~f_1(x_2)=-qa_0^{-1}x_1x_2,~~f_1(x_3)=-qa_0^{-1}x_1x_3,\\
&&f_2(x_1)=0,~~f_2(x_2)=0,~~f_2(x_3)=f_0,\end{eqnarray*}
where $a_0$, $f_0\in \mathbb{C}\backslash\{0\}$.

All these structures are isomorphic to that with $a_0=f_0=1$.
\end{lemma}
\begin{proof}
 By Lemma \ref{10} and Lemma \ref{S1}, to determine the module-algebra structures of $U_q(sl(3))$ on $A_q(3)$, we have to make (\ref{e11})-(\ref{e12}) hold for any $u\in\{x_1, x_2, x_3\}$ using the actions of $k_1$, $e_1$, $f_1$ in Lemma \ref{10} and the actions of $k_2$, $e_2$, $f_2$ in Lemma \ref{S1}.

Since $k_1e_2(x_1)=q^{-1}e_2k_1(x_1)=q^{-3}e_2(x_1)$, we have $e_2(x_1)=-qf_0^{-1}x_1x_3$, i.e., $\mu_1=\mu_2=\mu_3=0$. Using $k_1e_2(x_2)=q^{-1}e_2k_1(x_2)=
q^{-2}e_2(x_2)$, we obtain $e_2(x_2)=-qf_0^{-1}x_2x_3$. Similarly, by $k_2f_1(x_2)=qf_1k_2(x_2)=q^2f_1(x_2)$ and $k_2f_1(x_3)=qf_1k_2(x_3)=q^3f_1(x_3)$, we get
$f_1(x_2)=-qa_0^{-1}x_1x_2$ and $f_1(x_3)=-qa_0^{-1}x_1x_3$. Then, it is easy to check that (\ref{e11})-(\ref{r10}) hold for any $u\in \{x_1, x_2, x_3\}$.

Then, we check that (\ref{r11}) holds. Obviously, $e_1f_2(u)=f_2e_1(u)$ for any $u\in \{x_1, x_2, x_3\}$. Now, we check $e_2f_1(x_1)-f_1e_2(x_1)=0$. In fact,
\begin{eqnarray*}
& &e_2f_1(x_1)-f_1e_2(x_1)\\
&=&e_2(-qa_0^{-1}x_1^2)-f_1(-qf_0^{-1}x_1x_3)\\
&=&-qa_0^{-1}(x_1e_2(x_1)+e_2(x_1)k_2(x_1))+qf_0^{-1}(k_1^{-1}(x_1)f_1(x_3)+f_1(x_1)x_3)\\
&=&(q^2+q^4)a_0^{-1}f_0^{-1}x_1^2x_3-(q^4+q^2)a_0^{-1}f_0^{-1}x_1^2x_3\\
&=&0.
\end{eqnarray*}
Similarly, other equalities in (\ref{r11}) can be checked.

Next, we check that (\ref{r12})-(\ref{e12}) hold for any $u\in \{x_1, x_2, x_3\}$.
Here, we only check $e_2^2e_1(x_1)-(q+q^{-1})e_2e_1e_2(x_1)+e_1e_2^2(x_1)=0$. The other equalities can be checked similarly. In fact,
\begin{eqnarray*}
& & e_2^2e_1(x_1)-(q+q^{-1})e_2e_1e_2(x_1)+e_1e_2^2(x_1)\\
&=&0-(q+q^{-1})e_2e_1(-qf_0^{-1}x_1x_3)+e_1e_2(-qf_0^{-1}x_1x_3)\\
&=&(q+q^{-1})f_0^{-1}e_2(a_0x_3)-qf_0^{-1}e_1(-qf_0^{-1}x_1x_3^2-q^3f_0^{-1}x_1x_3^2)\\
&=&-q(q+q^{-1})a_0f_0^{-2}x_3^2+a_0f_0^{-2}(1+q^2)x_3^2\\
&=&0.
\end{eqnarray*}

Finally, we claim that all the actions with nonzero $a_0$ and $f_0$ are isomorphic to the specific action with $a_0=1$, $f_0=1$.
The desired isomorphism is given by $\psi_{a_0, f_0}: x_1\mapsto a_0 x_1, x_2\mapsto x_2, x_3\mapsto f_0x_3$.
\end{proof}

\begin{lemma}\label{S3}
For Case $(A2)|(B4)$, all $U_q(sl(3))$-module algebra structures on $A_q(3)$ are as follows
\begin{eqnarray*}
&&k_1(x_1)=q^{-2} x_1, ~~k_1(x_2)=q^{-1} x_2,~~k_1(x_3)=q^{-1} x_3,\\
&&k_2(x_1)=qx_1, ~~k_2(x_2)=q^{-1} x_2,~~k_2(x_3)=x_3,\\
&&e_1(x_1)=a_0,~~e_1(x_2)=0,~~e_1(x_3)=0,\\
&&e_2(x_1)=0,~~e_2(x_2)=a_2x_1,~~e_2(x_3)=0,\\
&&f_1(x_1)=-qa_0^{-1}x_1^2,~~f_1(x_2)=-qa_0^{-1}x_1x_2,~~f_1(x_3)=-qa_0^{-1}x_1x_3,\\
&&f_2(x_1)=a_2^{-1}x_2,~~f_2(x_2)=0,~~f_2(x_3)=0,
\end{eqnarray*}
where $a_0$, $a_2\in \mathbb{C}\backslash\{0\}$.

All these module-algebra structures are isomorphic to that with $a_0=a_2=1$.
\end{lemma}
\begin{proof}
 By the above actions of $k_1$, $e_1$, $f_1$ and $k_2$, $e_2$, $f_2$, it is easy to check that (\ref{e11})-(\ref{e12}) hold for any $u\in\{x_1, x_2, x_3\}$. Therefore,
by Lemma \ref{10} and Lemma \ref{t1}, they determine the module-algebra structures of $U_q(sl(3))$ on $A_q(3)$.

Next, we prove that there are no other actions except these in this lemma.

Using $k_1e_2(x_1)=q^{-1}e_2k_1(x_1)=q^{-3}e_2(x_1)$, we can obtain  $e_2(x_1)=0$. By Lemma \ref{t1}, we also have $e_2(x_2)=a_2x_1$ and $e_2(x_3)=0$.
Similarly, by $k_1f_2(x_1)=qf_2k_1(x_1)=q^{-1}f_2(x_1)$, we get $f_2(x_1)=c_1x_2$. Therefore, $f_2(x_2)=f_2(x_3)=0$. Then, using $e_2f_2(x_i)-f_2e_2(x_i)=\frac{k_2-k_2^{-1}}{q-q^{-1}}(x_i)$ for any $i\in\{1, 2, 3\}$, we obtain $c_1=a_2^{-1}$. Since $k_2f_1(x_2)=qf_1k_2(x_2)=f_1(x_2)$ and $k_2f_1(x_3)=qf_1k_2(x_3)=qf_1(x_3)$, by Lemma \ref{10}, we have $f_1(x_2)=-qa_0^{-1}x_1x_2+\xi_3x_3^3$, $f_1(x_3)=-qa_0^{-1}x_1x_3$.

Due to the condition of module algebra, it is easy to see that we have to let $f_1^2f_2(x_3)-(q+q^{-1})f_1f_2f_1(x_3)+f_2f_1^2(x_3)=0$ hold. On the other hand, we have
\begin{eqnarray*}
&&f_1^2f_2(x_3)-(q+q^{-1})f_1f_2f_1(x_3)+f_2f_1^2(x_3)\\
&=&-(q+q^{-1})f_1f_2(-qa_0^{-1}x_1x_3)+f_2f_1(-qa_0^{-1}x_1x_3)\\
&=&q(q+q^{-1})a_0^{-1}a_2^{-1}f_1(x_2x_3)-qa_0^{-1}f_2(f_1(x_1)x_3+q^2x_1f_1(x_3))\\
&=&q(q+q^{-1})a_0^{-1}a_2^{-1}(f_1(x_2)x_3+qx_2f_1(x_3))+qa_0^{-2}(q+q^3)f_2(x_1^2x_3)\\
&=&(q^2+1)a_0^{-1}a_2^{-1}(-(q+q^3)a_0^{-1}x_1x_2x_3+\xi_3x_3^4)\\
&&+(q^2+1)(q+q^3)a_2^{-1}a_0^{-2}x_1x_2x_3\\
&=&(q^2+1)a_0^{-1}a_2^{-1}\xi_3x_3^4.
\end{eqnarray*}
Hence, we get $\xi_3=0$. Therefore, $f_1(x_2)=-qa_0^{-1}x_1x_2$.

Therefore, there are no other actions except these in this lemma.

Finally, we claim that all the actions with nonzero $a_0$ and $a_2$ are isomorphic to the specific action with $a_0=1$, $a_2=1$.
The desired isomorphism is given by $\psi_{a_0, a_2}: x_1\mapsto a_0 x_1, x_2\mapsto a_0a_2x_2, x_3\mapsto x_3$.
\end{proof}

\begin{lemma}\label{S4}
For Case $(A2)|(B7)$, all module-algebra structures of $U_q(sl(3))$ on $A_q(3)$ are as follows
\begin{eqnarray*}
&&k_1(x_1)=q^{-2}x_1,~~k_1(x_2)=q^{-1}x_2,~~k_1(x_3)=q^{-1}x_3,\\
&&k_2(x_1)=qx_1,~~k_2(x_2)=q^2x_2,~~k_2(x_3)=q^{-1}x_3,\\
&&e_1(x_1)=a_0,~~e_1(x_2)=0,~~e_1(x_3)=0,\\
&&e_2(x_1)=-qe_0^{-1}x_1x_2,~~e_2(x_2)=-qe_0^{-1}x_2^2,~~e_2(x_3)=a_3x_1+e_0^{-1}x_2x_3,\\
&&f_1(x_1)=-qa_0^{-1}x_1^2,~~f_1(x_2)=-qa_0^{-1}x_1x_2,~~f_1(x_3)=-qa_0^{-1}x_1x_3+\xi_4x_2x_3^2,\\
&&f_2(x_1)=0,~~f_2(x_2)=e_0,~~f_2(x_3)=0,
\end{eqnarray*}
where $a_0$, $e_0$, $a_3$, $\xi_4\in\mathbb{C}\backslash\{0\}$ and $a_3\xi_4=-qe_0^{-1}a_0^{-1}$.

All module-algebra structures above are isomorphic to that with $a_0=e_0=a_3=1$ and $\xi_4=-q$.
\end{lemma}
\begin{proof}
 By the above actions of $k_1$, $e_1$, $f_1$ and $k_2$, $e_2$, $f_2$, it is easy to check that (\ref{e11})-(\ref{e12}) hold for any $u\in\{x_1, x_2, x_3\}$. Therefore,
by Lemma \ref{10} and Lemma \ref{S2}, they determine the module-algebra structures of $U_q(sl(3))$ on $A_q(3)$.

Next, we prove that there are no other actions except these in this lemma.

By (\ref{e11}), we can immediately obtain that $e_2(x_1)=\widetilde{\alpha_0}x_1x_2$, $e_2(x_2)=\widetilde{\beta_0}x_2^2$, $e_2(x_3)=a_3x_1+\widetilde{\gamma_0}x_2x_3$,
$f_2(x_1)=0$, $f_2(x_2)=e_0$ and $f_2(x_3)=0$. By Lemma \ref{S2}, $\widetilde{\alpha_0}=\widetilde{\beta_0}=-q\widetilde{\gamma_0}$. According to
$e_2f_2(x_1)-f_2e_2(x_1)=\frac{k_2-k_2^{-1}}{q-q^{-1}}(x_1)$, we obtain $\widetilde{\gamma_0}=e_0^{-1}$. Similarly, by (\ref{r10}), we have $f_1(x_1)=-qa_0^{-1}x_1^2$, $f_1(x_2)=-qa_0^{-1}x_1x_2$ and $f_1(x_3)=-qa_0^{-1}x_1x_3+\xi_4x_2x_3^2$.

Next, let us consider the condition $e_2f_1(x_3)-f_1e_2(x_3)=0$. Since\\
$e_2f_1(x_3)-f_1e_2(x_3)$
\begin{eqnarray*}
&=&e_2(-qa_0^{-1}x_1x_3+\xi_4x_2x_3^2)-f_1(a_3x_1+e_0^{-1}x_2x_3)\\
&=&-qa_0^{-1}(x_1e_2(x_3)+e_2(x_1)k_2(x_3))+\xi_4(x_2x_3e_2(x_3)+x_2e_2(x_3)k_2(x_3)\\
&&+e_2(x_2)k_2(x_3)k_2(x_3))+qa_3a_0^{-1}x_1^2-e_0^{-1}(k_1^{-1}(x_2)f_1(x_3)+f_1(x_2)x_3)\\
&=&\xi_4a_3(q^2+1)x_1x_2x_3+a_0^{-1}e_0^{-1}(q^3+q)x_1x_2x_3\\
&=&0,
\end{eqnarray*}
we obtain $\xi_4a_3=-qa_0^{-1}e_0^{-1}$.

Therefore, there are no other actions except these in this lemma.

Finally, we show that all the actions with nonzero $a_0$, $a_3$, $e_0$ and $\xi_4$ are isomorphic to the specific action with $a_0=e_0=a_3=1$
and $\xi_4=-q$.
The desired isomorphism is given by $\psi_{a_0, a_3, e_0}: x_1\mapsto a_0 x_1, x_2\mapsto e_0x_2, x_3\mapsto a_0a_3x_3$.

\end{proof}
\begin{lemma}\label{r17}
For Case $(A3)|(B5)$, all module-algebra structures of $U_q(sl(3))$ on $A_q(3)$ are as follows:
\begin{eqnarray*}
&&k_1(x_1)=qx_1,~~k_1(x_2)=qx_2,~~k_1(x_3)=q^2x_3,\\
&&k_2(x_1)=x_1,~~k_2(x_2)=qx_2,~~k_2(x_3)=q^{-1}x_3,\\
&&e_1(x_1)=-qf_0^{-1}x_1x_3,~~e_1(x_2)=-qf_0^{-1}x_2x_3,~~e_1(x_3)=-qf_0^{-1}x_3^2,\\
&&e_2(x_1)=0,~~e_2(x_2)=0,~~e_2(x_3)=b_3x_2,\\
&&f_1(x_1)=0,~~f_1(x_2)=0,~~f_1(x_3)=f_0,\\
&&f_2(x_1)=0,~~f_2(x_2)=b_3^{-1}x_3,~~f_2(x_3)=0,
\end{eqnarray*}
where $b_3$, $f_0\in\mathbb{C}\backslash\{0\}$.

All module-algebra structures above are isomorphic to that with $b_3=f_0=1$.
\end{lemma}
\begin{proof}
The proof is similar to that in Lemma \ref{S3}.
\end{proof}

\begin{lemma}\label{r18}
For Case $(A3)|(B6)$, all module-algebra structures of $U_q(sl(3))$ on $A_q(3)$ are as follows:
\begin{eqnarray*}
&&k_1(x_1)=qx_1,~~k_1(x_2)=qx_2,~~k_1(x_3)=q^2x_3,\\
&&k_2(x_1)=qx_1,~~k_2(x_2)=q^{-2}x_2,~~k_2(x_3)=q^{-1}x_3,\\
&&e_1(x_1)=-qf_0^{-1}x_1x_3+\mu_1x_1^2x_2,~~e_1(x_2)=-qf_0^{-1}x_2x_3,~~e_1(x_3)=-qf_0^{-1}x_3^2,\\
&&e_2(x_1)=0,~~e_2(x_2)=b_0,~~e_2(x_3)=0,\\
&&f_1(x_1)=0,~~f_1(x_2)=0,~~f_1(x_3)=f_0,\\
&&f_2(x_1)=d_1x_3+b_0^{-1}x_1x_2,~~f_2(x_2)=-qb_0^{-1}x_2^2,~~f_2(x_3)=-qb_0^{-1}x_2x_3,
\end{eqnarray*}
where $d_1$, $b_0$, $\mu_1$, $f_0\in\mathbb{C}\backslash\{0\}$ and $\mu_1d_1=-qb_0^{-1}f_0^{-1}$.

All module-algebra structures above are isomorphic to that with $d_1=b_0=f_0=1$ and $\mu_1=-q$.
\end{lemma}
\begin{proof}
The proof is similar to that in Lemma \ref{S4}.
\end{proof}

\begin{lemma}
For Case $(A4)|(B5)$, all module-algebra structures of $U_q(sl(3))$ on $A_q(3)$ are as follows:
\begin{eqnarray*}
&&k_1(x_1)=qx_1,~~k_1(x_2)=q^{-1}x_2,~~k_1(x_3)=x_3,\\
&&k_2(x_1)=x_1,~~k_2(x_2)=qx_2,~~k_2(x_3)=q^{-1}x_3,\\
&&e_1(x_1)=0,~~e_1(x_2)=a_2x_1,~~e_1(x_3)=0,\\
&&e_2(x_1)=0,~~e_2(x_2)=0,~~e_2(x_3)=b_3x_2,\\
&&f_1(x_1)=a_2^{-1}x_2,~~f_1(x_2)=0,~~f_1(x_3)=0,\\
&&f_2(x_1)=0,~~f_2(x_2)=b_3^{-1}x_3,~~f_2(x_3)=0,
\end{eqnarray*}
where $a_2$, $b_3\in \mathbb{C}\backslash\{0\}$.

All the above module-algebra structures are isomorphic to that with $a_2=b_3=1$.

\end{lemma}
\begin{proof}
By the actions of $k_1$, $e_1$, $f_1$ and $k_2$, $e_2$, $f_2$, it is easy to check that (\ref{e11})-(\ref{e12}) hold for any $u\in\{x_1, x_2, x_3\}$. Therefore,
by Lemma \ref{t1} and Lemma \ref{tt1}, they determine the module-algebra structures of $U_q(sl(3))$ on $A_q(3)$.

Next, we prove that there are no other actions except these in this lemma.

By Lemma \ref{tt1} and using (\ref{e11}) holding for any $u\in \{x_1, x_2, x_3\}$, we can obtain that
$e_2(x_1)=\sum_{n\geq 0}\widetilde{a_n}x_1^{n+2}x_2^{n+2}x_3^{n}$, $e_2(x_2)=0$, $e_2(x_3)=b_3x_2+\sum_{n\geq 0}\widetilde{c_n}x_1^{n+1}x_2^{n+2}$\\
$\cdot x_3^{n+1}$,
$f_2(x_1)=\sum_{m\geq 0}\widetilde{d_{m,m+1}}x_1^{m+2}x_2^mx_3^{m+2}$, $f_2(x_2)=d_2x_3+\sum_{m\geq 0}\widetilde{e_{m,m+1}}x_1^{m+1}$\\
$\cdot x_2^{m+1}x_3^{m+2}$,
$f_2(x_3)=\sum_{m\geq 0}\widetilde{g_{m,m+1}}x_1^{m+1}x_2^{m}x_3^{m+3}$. By Lemma \ref{tt1}, we know that $\frac{\widetilde{a_n}}{\widetilde{c_n}}=
-\frac{q(2n+2)_q}{(2n+4)_q}$, $\frac{\widetilde{d_{m,m+1}}}{\widetilde{e_{m,m+1}}}=\frac{(2m+2)_q}{q^{m}(2)_q}$,  $\frac{\widetilde{d_{m,m+1}}}{\widetilde{g_{m,m+1}}}=\frac{q^{2m+2}-1}{1-q^{2m+2}}=-1$. Set $v_n=\frac{\widetilde{a_n}}{\widetilde{c_n}}$
and $\kappa_m=\frac{\widetilde{d_{m,m+1}}}{\widetilde{e_{m,m+1}}}$.

Next, we consider $e_2f_2(x_2)-f_2e_2(x_2)=\frac{k_2-k_2^{-1}}{q-q^{-1}}(x_2)=x_2$. By some computations, we obtain\\
$e_2f_2(x_2)-f_2e_2(x_2)$
\begin{eqnarray*}
&=&b_3d_2x_2+\sum_{n\geq 0}(q^{2n+4}-1)d_2v_nx_1^{n+1}x_2^{n+2}x_3^{n+1}-\sum_{m\geq0}(q^{-1}-q^{2m+3})\\
&&\cdot b_3\kappa_mx_1^{m+1}x_2^{m+2}x_3^{m+1}+\sum_{m\geq 0,n\geq 0}q^{3mn+3m+3n+2}(1-q^{2})\\
&&\cdot(1-q^{2m+2n+6})v_n\kappa_mx_1^{m+n+2}x_2^{m+n+3}x_3^{m+n+2}.
\end{eqnarray*}
If there exist $v_n$ and $\kappa_m$ not equal to zero, we can choose the terms with coefficients $v_{n_e}$ and $\kappa_{m_f}$ in $e_2(x_1)$, $e_2(x_2)$, $e_2(x_3)$, $f_2(x_1)$, $f_2(x_2)$, $f_2(x_3)$ such that the degrees of them are the highest. Then, the unique monomial of the highest degree
in $(e_2f_2-f_2e_2)$$(x_2)$ is \begin{eqnarray*}q^{3m_fn_e+3m_f+3n_e+2}(1-q^{2})(1-q^{2m_f+2n_e+6})v_{n_e}\kappa_{m_f}\\
x_1^{m_f+n_e+2}x_2^{m_f+n_e+3}x_3^{m_f+n_e+2}.\end{eqnarray*} Since the degree of this term is larger than 1, this case is impossible. Similarly, all cases except that all $v_n$, $\kappa_m$ are equal to zero should be excluded.
Therefore, we obtain that $e_2(x_1)=0$, $e_2(x_2)=0$, $e_2(x_3)=b_3x_2$, $f_2(x_1)=0$, $f_2(x_2)=b_3^{-1}x_3$ and $f_2(x_3)=0$.

Similarly, using (\ref{r10}), Lemma \ref{t1} and $e_1f_1(u)-f_1e_1(u)=\frac{k_1-k_1^{-1}}{q-q^{-1}}(u)$ for any $u\in\{x_1, x_2, x_3\}$, we can get
$e_1(x_1)=0$, $e_1(x_2)=a_2x_1$, $e_1(x_3)=0$, $f_1(x_1)=a_2^{-1}x_2$, $f_1(x_2)=0$, $f_1(x_3)=0$.

Therefore, there are no other actions except ones in this lemma.

Finally, we claim that all the actions with nonzero $a_2$, $b_3$ are isomorphic to the specific action with $a_2=b_3=1$.
The desired isomorphism is given by $\psi_{a_2, b_3}: x_1\mapsto x_1, x_2\mapsto a_2x_2, x_3\mapsto a_2b_3x_3$.

\end{proof}
\begin{lemma}\label{S5}
For Case $(A5)|(B6)$, all module-algebra structures of $U_q(sl(3))$ on $A_q(3)$ are
\begin{eqnarray*}
&&k_1(x_1)=x_1,~~k_1(x_2)=qx_2,~~k_1(x_3)=q^{-1}x_3,\\
&&k_2(x_1)=qx_1,~~k_2(x_2)=q^{-2}x_2,~~k_2(x_3)=q^{-1}x_3,\\
&&e_1(x_1)=0,~~e_1(x_2)=0,~~e_1(x_3)=b_3x_2,\\
&&e_2(x_1)=0,~~e_2(x_2)=b_0,~~e_2(x_3)=0,\\
&&f_1(x_1)=0,~~f_1(x_2)=b_3^{-1}x_3,~~f_1(x_3)=0,\\
&&f_2(x_1)=b_0^{-1}x_1x_2,~~f_2(x_2)=-qb_0^{-1}x_2^2,~~f_2(x_3)=-qb_0^{-1}x_2x_3,
\end{eqnarray*}
where $b_3$, $b_0\in \mathbb{C}\backslash\{0\}$.

All module-algebra structures are isomorphic to that with $b_0=b_3=1$.
\end{lemma}
\begin{proof}
It is easy to check that the above actions of $k_1$, $e_1$, $f_1$ and $k_2$, $e_2$, $f_2$ determine module-algebra structures of $U_q(sl(3))$ on $A_q(3)$.

Then, we prove that there are no other actions except these in this lemma.

By (\ref{e11}) for any $u\in\{x_1, x_2, x_3\}$ and Lemma \ref{t3}, we have
\begin{eqnarray*}
&&e_2(x_1)=(q-q^3)ux_1^4x_3+\sum_{n\geq 0}(q-q^{2n+5})v_nx_1^{3n+7}x_2^{n+1}x_3^{n+2},\\
&&e_2(x_2)=b_0+(q^4-1)ux_1^3x_2x_3+\sum_{n\geq 0}(q^{3n+7}-q^{n+1})v_nx_1^{3n+6}x_2^{n+2}x_3^{n+2},\\
&&e_2(x_3)=(q^4-1)ux_1^3x_3^2+\sum_{n\geq 0}(q^{4n+8}-1)v_nx_1^{3n+6}x_2^{n+1}x_3^{n+3},\\
&&f_2(x_1)=gx_1x_2+(q^3-q^{-1})\varepsilon x_1^4x_2^2x_3+\sum_{p\geq 0}(q^{2p+5}-q^{-1})\mu_px_1^{3p+7}x_2^{p+3}x_3^{p+2},\\
&&f_2(x_2)=-qgx_2^2+(q-q^5)\varepsilon x_1^3x_2^3x_3+\sum_{p\geq 0}(q^{p+2}-q^{3p+8})\mu_px_1^{3p+6}x_2^{p+4}x_3^{p+2},\\
&&f_2(x_3)=-qgx_2x_3+(1-q^6)\varepsilon x_1^{3}x_2^{2}x_3^2+\sum_{p\geq 0}(1-q^{4p+10})\mu_px_1^{3p+6}x_2^{p+3}x_3^{p+3}.
\end{eqnarray*}

Then, we consider the condition $e_2f_2(u)-f_2e_2(u)=\frac{k_2-k_2^{-1}}{q-q^{-1}}(u)$ for any $u\in\{x_1, x_2, x_3\}$.

Let us assume that there exist some $u$ or $v_n$ which are not equal to zero. Then, we can choose the monomials in $e_2(x_1)$, $e_2(x_2)$, $e_2(x_3)$ with the highest degree. Obviously, these monomials are unique. It is also easy to see that $f(x_1)$, $f(x_2)$, $f(x_3)$ can not be equal to zero simultaneously. Therefore, there are some nonzero $g$, $\varepsilon$ or $\mu_p$. Similarly, those monomials in $f_2(x_1)$, $f_2(x_2)$, $f_2(x_3)$ with the highest degree are chosen. Then, by some computations, we can obtain a monomial with the highest degree, whose degree is larger than 1. Then, we get a contradiction with $e_2f_2(x_1)-f_2e_2(x_1)=x_1$. For example, if the coefficient of the monomials in $e_2(x_1)$, $e_2(x_2)$ and $e_2(x_3)$ is $v_{n_e}$ and the coefficient of the monomials in $f_2(x_1)$, $f_2(x_2)$, $f_2(x_3)$ with the highest degree is $\mu_{p_f}$, then the monomial with the highest degree in $e_2f_2(x_1)-f_2e_2(x_1)$ is\\
$$q^{7n_ep_f+15n_e+11p_f+22}(1-q^{2n_e+2p_f+10})^2v_{n_e}\mu_{p_f}x_1^{3n_e+3p_f+13}x_2^{n_e+p_f+4}x_3^{n_e+p_f+4}.$$

Therefore, all $u$, $v_n$ are equal to zero. Then, $e_2(x_1)=0$, $e_2(x_2)=b_0$ and $e_2(x_3)=0$.
Thus, we can obtain
\begin{eqnarray*}
&&e_2f_2(x_1)-f_2e_2(x_1)\\
&=&gb_0x_1+(1-q^{-4})(1+q^2)\varepsilon b_0x_1^4x_2x_3\\
&&+\sum_{p\geq 0}\frac{1-q^{2p+6}}{1-q^2}(q^{-1-p}-q^{-3p-7})b_0\mu_px_1^{3p+7}x_2^{p+2}x_3^{p+2}.
\end{eqnarray*}
Thus, we obtain $f_2(x_1)=b_0^{-1}x_1x_2$, $f_2(x_2)=-qb_0^{-1}x_2^2$ and $f_2(x_3)=-qb_0^{-1}x_2x_3$.

On the other hand, with a similar discussion, by (\ref{r10}), Lemma \ref{tt1} and $e_1f_1(u)-f_1e_1(u)=\frac{k_1-k_1^{-1}}{q-q^{-1}}(u)$ for any $u\in\{x_1, x_2, x_3\}$, we can obtain $e_1(x_1)=0$, $e_1(x_2)=0$, $e_1(x_3)=b_3x_2$,
$f_1(x_1)=0$, $f_1(x_2)=b_3^{-1}x_3$, $f_1(x_3)=0$.

Therefore, there are no other actions except these in this lemma.

Finally, we claim that all module algebra structures with nonzero $b_0$, $b_3$ are isomorphic to that with $b_0=b_3=1$. The desired isomorphism is given by $$\psi_{b_0, b_3}:x_1\mapsto x_1, x_2\mapsto b_0x_2, x_3\mapsto b_0b_3x_3.$$
\end{proof}
\begin{lemma}\label{SS2}
For Case $(A4)|(B7)$, all module-algebra structures of $U_q(sl(3))$ on $A_q(3)$ are as follows
\begin{eqnarray*}
&&k_1(x_1)=qx_1,~~k_1(x_2)=q^{-1}x_2,~~k_1(x_3)=x_3,\\
&&k_2(x_1)=qx_1,~~k_2(x_2)=q^2x_2,~~k_2(x_3)=q^{-1}x_3,\\
&&e_1(x_1)=0,~~e_1(x_2)=a_2x_1,~~e_1(x_3)=0,\\
&&e_2(x_1)=-qe_0^{-1}x_1x_2,~~e_2(x_2)=-qe_0^{-1}x_2^2,~~e_2(x_3)=e_0^{-1}x_2x_3,\\
&&f_1(x_1)=a_2^{-1}x_2,~~f_1(x_2)=0,~~f_1(x_3)=0,\\
&&f_2(x_1)=0,~~f_2(x_2)=e_0,~~f_2(x_3)=0,
\end{eqnarray*}
where $a_2$, $e_0\in \mathbb{C}\backslash\{0\}$.

All module-algebra structures are isomorphic to that with $a_2=e_0=1$.
\end{lemma}
\begin{proof}
The proof is similar to that in Lemma \ref{S5}.
\end{proof}
Since the actions of $k_1$, $e_1$, $f_1$ and $k_2$, $e_2$, $f_2$ in $U_q(sl(3))$ are symmetric, by Lemma \ref{SS1}-Lemma \ref{SS2} and the discussion above, we obtain the following theorem.
\begin{theorem}\label{t10}
$U_q(sl(3))$-symmetries up to isomorphisms on $A_q(3)$ and their classical limits, i.e., Lie algebra $sl_3$-actions by differentiations on $\mathbb{C}[x_1,x_2,x_3]$ are as follows:\\
\begin{tabular}{|l|l|}
\hline $U_q(sl(3))$-symmetries &  Classical limit \\
& $sl_3$-actions on $\mathbb{C}[x_1,x_2,x_3]$ \\
\hline
$k_i(x_1)=\pm x_1$, $k_i(x_2)=\pm x_2$,  & $h_i(x_1)=0$, $h_i(x_2)=0$,   \\
$k_i(x_3)=\pm x_3,$ $k_j(x_1)=\pm x_1$, &$h_i(x_3)=0,$  $h_j(x_1)=0$,  \\
$k_j(x_2)=\pm x_2$, $k_j(x_3)=\pm x_3,$ &$h_j(x_2)=0$, $h_j(x_3)=0,$\\
$e_i(x_1)=e_i(x_2)=e_i(x_3)=0$,  & $e_i(x_1)=e_i(x_2)=e_i(x_3)=0$,\\
$f_i(x_1)=f_i(x_2)=f_i(x_3)=0,$ & $f_i(x_1)=f_i(x_2)=f_i(x_3)=0,$ \\
$e_j(x_1)=e_j(x_2)=e_j(x_3)=0$, & $e_j(x_1)=e_j(x_2)=e_j(x_3)=0$, \\
 $f_j(x_1)=f_j(x_2)=f_j(x_3)=0$ & $f_j(x_1)=f_j(x_2)=f_j(x_3)=0$ \\
\hline
$k_i(x_1)=q^{-2} x_1$, $k_i(x_2)=q^{-1} x_2$,  & $h_i(x_1)=-2x_1$, $h_i(x_2)=-x_2$, \\
$k_i(x_3)=q^{-1} x_3$, $k_j(x_1)=qx_1$, & $h_i(x_3)=-x_3$, $h_j(x_1)=x_1$,\\
 $k_j(x_2)=q x_2$, $k_j(x_3)=q^2 x_3,$ &  $h_j(x_2)=x_2$, $h_j(x_3)=2x_3,$\\
$e_i(x_1)=1$, $e_i(x_2)=0$, $e_i(x_3)=0,$ & $e_i(x_1)=1$, $e_i(x_2)=0$, $e_i(x_3)=0,$\\
$e_j(x_1)=-qx_1x_3$, $e_j(x_2)=-qx_2x_3$,& $e_j(x_1)=-x_1x_3$, $e_j(x_2)=-x_2x_3$, \\
$e_j(x_3)=-qx_3^2,$ $f_i(x_1)=-qx_1^2$, & $e_j(x_3)=-x_3^2,$ $f_i(x_1)=-x_1^2$,   \\
 $f_i(x_2)=-qx_1x_2$, $f_i(x_3)=-qx_1x_3,$ & $f_i(x_2)=-x_1x_2$, $f_i(x_3)=-x_1x_3,$ \\
$f_j(x_1)=0$, $f_j(x_2)=0$, $f_j(x_3)=1$ & $f_j(x_1)=0$, $f_j(x_2)=0$, $f_j(x_3)=1$\\
\hline
$k_i(x_1)=q^{-2} x_1$, $k_i(x_2)=q^{-1} x_2$,  & $h_i(x_1)=-2 x_1$, $h_i(x_2)=-x_2$,\\
$k_i(x_3)=q^{-1} x_3,$ $k_j(x_1)=qx_1$, &  $h_i(x_3)=-x_3,$ $h_j(x_1)=x_1$,\\
$k_j(x_2)=q^{-1} x_2$, $k_j(x_3)=x_3,$ &  $h_j(x_2)=- x_2$, $h_j(x_3)=0,$ \\
$e_i(x_1)=1$, $e_i(x_2)=0$, $e_i(x_3)=0,$ & $e_i(x_1)=1$, $e_i(x_2)=0$, $e_i(x_3)=0,$ \\
$e_j(x_1)=0$, $e_j(x_2)=x_1$, $e_j(x_3)=0,$ & $e_j(x_1)=0$, $e_j(x_2)=x_1$, $e_j(x_3)=0,$  \\
$f_i(x_1)=-qx_1^2$, $f_i(x_2)=-qx_1x_2$,  & $f_i(x_1)=-x_1^2$, $f_i(x_2)=-x_1x_2$, \\
$f_i(x_3)=-qx_1x_3,$ & $f_i(x_3)=-x_1x_3,$ \\
$f_j(x_1)=x_2$, $f_j(x_2)=0,$ $f_j(x_3)=0$ & $f_j(x_1)=x_2$, $f_j(x_2)=0,$ $f_j(x_3)=0$ \\
\hline
$k_i(x_1)=q x_1$, $k_i(x_2)=qx_2$, & $h_i(x_1)=x_1$, $h_i(x_2)=x_2$,  \\
 $k_i(x_3)=q^{2} x_3$, $k_j(x_1)=x_1$,& $h_i(x_3)=2x_3$, $h_j(x_1)=0$,  \\
  $k_j(x_2)=qx_2$, $k_j(x_3)=q^{-1}x_3,$ & $h_j(x_2)=x_2$, $h_j(x_3)=-x_3,$\\
$e_i(x_1)=-qx_1x_3$, $e_i(x_2)=-qx_2x_3$,  & $e_i(x_1)=-x_1x_3$, $e_i(x_2)=-x_2x_3$, \\
$e_i(x_3)=-qx_3^2,$ $e_j(x_1)=0$, & $e_i(x_3)=-x_3^2,$ $e_j(x_1)=0$, \\
$e_j(x_2)=0$, $e_j(x_3)=x_2,$ & $e_j(x_2)=0$, $e_j(x_3)=x_2,$\\
$f_i(x_1)=0$, $f_i(x_2)=0$, $f_i(x_3)=1,$ & $f_i(x_1)=0$, $f_i(x_2)=0$, $f_i(x_3)=1,$  \\
$f_j(x_1)=0$, $f_j(x_2)=x_3$, $f_j(x_3)=0$ & $f_j(x_1)=0$, $f_j(x_2)=x_3$, $f_j(x_3)=0$  \\
\hline
$k_i(x_1)=q x_1$, $k_i(x_2)=q^{-1}x_2$, & $h_i(x_1)= x_1$, $h_i(x_2)=-x_2$, \\
$k_i(x_3)=x_3,$ $k_j(x_1)=x_1$,& $h_i(x_3)=0,$ $h_j(x_1)=0$,  \\
 $k_j(x_2)=qx_2$, $k_j(x_3)=q^{-1}x_3,$ &$h_j(x_2)=x_2$, $h_j(x_3)=-x_3,$ \\
$e_i(x_1)=0$, $e_i(x_2)=x_1$, $e_i(x_3)=0,$ & $e_i(x_1)=0$, $e_i(x_2)=x_1$, $e_i(x_3)=0,$\\
$e_j(x_1)=0$, $e_j(x_2)=0$, $e_j(x_3)=x_2,$ & $e_j(x_1)=0$, $e_j(x_2)=0$, $e_j(x_3)=x_2,$\\
$f_i(x_1)=x_2$, $f_i(x_2)=0$, $f_i(x_3)=0,$ & $f_i(x_1)=x_2$, $f_i(x_2)=0$, $f_i(x_3)=0,$ \\
$f_j(x_1)=0$, $f_j(x_2)=x_3$, $f_j(x_3)=0$ & $f_j(x_1)=0$, $f_j(x_2)=x_3$, $f_j(x_3)=0$ \\
\hline
$k_i(x_1)=q x_1$, $k_i(x_2)=q^{-1}x_2$, & $h_i(x_1)=x_1$, $h_i(x_2)=-x_2$, \\
$k_i(x_3)=x_3,$ $k_j(x_1)=qx_1 $, & $h_i(x_3)=0,$ $h_j(x_1)=x_1 $, \\
$k_j(x_2)=q^2x_2$, $k_j(x_3)=q^{-1}x_3,$ & $h_j(x_2)=2x_2$, $h_j(x_3)=-x_3,$\\
$e_i(x_1)=0$, $e_i(x_2)=x_1$, $e_i(x_3)=0,$ & $e_i(x_1)=0$, $e_i(x_2)=x_1$, $e_i(x_3)=0,$ \\
$e_j(x_1)=-qx_1x_2$, $e_j(x_2)=-qx_2^2$,& $e_j(x_1)=-x_1x_2$, $e_j(x_2)=-x_2^2$,\\
 $e_j(x_3)=x_2x_3,$ $f_i(x_1)=x_2$,& $e_j(x_3)=x_2x_3,$  $f_i(x_1)=x_2$,  \\
 $f_i(x_2)=0$, $f_i(x_3)=0,$ &$f_i(x_2)=0$, $f_i(x_3)=0,$\\
$f_j(x_1)=0$, $f_j(x_2)=1$, $f_j(x_3)=0$ & $f_j(x_1)=0$, $f_j(x_2)=1$, $f_j(x_3)=0$ \\
\end{tabular}
\\
\begin{tabular}{|l|l|}
\hline
$k_i(x_1)=x_1$, $k_i(x_2)=qx_2$,& $h_i(x_1)=0$, $h_i(x_2)=x_2$, \\
 $k_i(x_3)=q^{-1}x_3,$ $k_j(x_1)=qx_1$,  & $h_i(x_3)=-x_3,$ $h_j(x_1)=x_1$,  \\
 $k_j(x_2)=q^{-2}x_2$,$k_j(x_3)=q^{-1}x_3,$  &  $h_j(x_2)=-2x_2$, $h_j(x_3)=-x_3,$\\
$e_i(x_1)=0$, $e_i(x_2)=0$, &  $e_i(x_1)=0$, $e_i(x_2)=0$, \\
$e_i(x_3)=x_2,$ $e_j(x_1)=0$,& $e_i(x_3)=x_2,$ $e_j(x_1)=0$,\\
 $e_j(x_2)=1$, $e_j(x_3)=0,$ &  $e_j(x_2)=1$, $e_j(x_3)=0,$\\
$f_i(x_1)=0$, $f_i(x_2)=x_3$, $f_i(x_3)=0,$ & $f_i(x_1)=0$, $f_i(x_2)=x_3$, $f_i(x_3)=0,$ \\
$f_j(x_1)=x_1x_2$, $f_j(x_2)=-qx_2^2$,  & $f_j(x_1)=x_1x_2$, $f_j(x_2)=-x_2^2$,\\
$f_j(x_3)=-qx_2x_3$ & $f_j(x_3)=-x_2x_3$  \\
\hline
$k_i(x_1)=q^{-2}x_1$, $k_i(x_2)=q^{-1}x_2$, & $h_i(x_1)=-2x_1$, $h_i(x_2)=-x_2$, \\
 $k_i(x_3)=q^{-1}x_3$, $k_j(x_1)=qx_1$, & $h_i(x_3)=-x_3,$ $h_j(x_1)=x_1$, \\
 $k_j(x_2)=q^2x_2$, $k_j(x_3)=q^{-1}x_3$, & $h_j(x_2)=2x_2$, $h_j(x_3)=-x_3,$\\
$e_i(x_1)=1$, $e_i(x_2)=0$, $e_i(x_3)=0$, &$e_i(x_1)=1$, $e_i(x_2)=0$, $e_i(x_3)=0$,\\
$e_j(x_1)=-qx_1x_2$, $e_j(x_2)=-qx_2^2$,& $e_j(x_1)=-x_1x_2$, $e_j(x_2)=-x_2^2$,\\
$e_j(x_3)=x_1+x_2x_3$, $f_i(x_1)=-qx_1^2,$ & $e_j(x_3)=x_1+x_2x_3$, $f_i(x_1)=-x_1^2,$\\
$f_i(x_2)=-qx_1x_2$,& $f_i(x_2)=-x_1x_2$, \\
 $f_i(x_3)=-qx_1x_3-qx_2x_3^2$,& $f_i(x_3)=-x_1x_3-x_2x_3^2$,\\
$f_j(x_1)=0$, $f_j(x_2)=1$, $f_j(x_3)=0$ & $f_j(x_1)=0$, $f_j(x_2)=1$, $f_j(x_3)=0$\\
\hline
$k_i(x_1)=qx_1$, $k_i(x_2)=qx_2$, & $h_i(x_1)=x_1$, $h_i(x_2)=x_2$,\\
 $k_i(x_3)=q^2x_3$, $k_j(x_1)=qx_1$, &  $h_i(x_3)=2x_3$, $h_j(x_1)=x_1$, \\
 $k_j(x_2)=q^{-2}x_2$, $k_j(x_3)=q^{-1}x_3$, & $h_j(x_2)=-2x_2$, $h_j(x_3)=-x_3$,\\
 $e_i(x_1)=-qx_1x_3-qx_1^2x_2$,  &   $e_i(x_1)=-x_1x_3-x_1^2x_2$, \\
$e_i(x_2)=-qx_2x_3$, $e_i(x_3)=-qx_3^2$,           & $e_i(x_2)=-yz$, $e_i(x_3)=-x_3^2$,\\
 $e_j(x_1)=0$, $e_j(x_2)=1$, $e_j(x_3)=0$,&  $e_j(x_1)=0$, $e_j(x_2)=1$, $e_j(x_3)=0$, \\
$f_i(x_1)=0$, $f_i(x_2)=0$, $f_i(x_3)=1$,& $f_i(x_1)=0$, $f_i(x_2)=0$, $f_i(x_3)=1$,\\
$f_j(x_1)=x_3+x_1x_2$, $f_j(x_2)=-qx_2^2$,& $f_j(x_1)=x_3+x_1x_2$, $f_j(x_2)=-x_2^2$,\\
$f_j(x_3)=-qx_2x_3$& $f_j(x_3)=-x_2x_3$\\
\hline
\end{tabular}\\
for any $i=1$, $j=2$ or $i=2$, $j=1$. Moreover, there are no isomorphisms between these nine kinds of module-algebra structures.
\end{theorem}

\begin{remark}
Case (5) when $i=1$, $j=2$ in Theorem \ref{t10} is the case discussed in \cite{hu2000} when $n=3$.
\end{remark}

Let us denote the actions of $U_q(sl(2))$ on $A_q(3)$ in (A1), those in (A2) and (B3) in Lemma \ref{r20}, those in (B4) in Lemma \ref{S3}, those in (B5) in Lemma \ref{r17}, those in (B6) in Lemma \ref{S5} and those in (B7) in Lemma \ref{SS2} by $\star1$, $\star2$, $\star3$, $\star4$, $\star5$, $\star6$, $\star7$ respectively. 
In addition, denote the actions of $U_q(sl(2))$ on $A_q(3)$ in (A2) and (B7) in Lemma \ref{S4}, those in (A3) and (B6) in Lemma \ref{r18} by $\star2'$, $\star7'$, $\star3'$,  $\star6'$ respectively. 
Then, as in Section 3, we can use the following diagrams to denote all actions of $U_q(sl(3))$ on $A_q(3)$:
\begin{eqnarray}\label{w4}
\xymatrix{{\star1}\ar@{-}[r]& {\star1}},~~~\xymatrix{{\star7'}\ar@{-}[r]&{\star2'}}, ~~
\xymatrix{{\star3'}\ar@{-}[r]\ar@{-}[r]& {\star6'}},\\
\label{r22}\xymatrix{ &{\star2}\ar@{-}[r]\ar@{-}[d]& {\star3}\ar@{-}[d]& \\
{\star7}\ar@{-}[r]&{\star4}\ar@{-}[r]&{\star5}\ar@{-}[r]&{\star6} }.
\end{eqnarray}
Just like that in Section 3, every two adjacent vertices corresponds to two classes of the module-algebra structures of $U_q(sl(3))$ on $A_q(3)$.

Then, we begin to study the module-algebra structures of $U_q(sl(m+1))$ on $A_q(3)$ for $m\geq 3$. For studying the module-algebra structures of $U_q(sl(m+1))$ on $A_q(3)$, we have to endow every vertex in the Dynkin diagram of $sl(m+1)$ an action of $U_q(sl(2))$ on $A_q(3)$.
As in Section 3, there are some rules which we should obey:
\begin{itemize}
\item[1.] Since every two adjacent vertices in the Dynkin diagram corresponds to one Hopf subalgebra isomorphic to $U_q(sl(3))$, by Theorem \ref{t10},
the action of $U_q(sl(2))$ on $A_q(3)$ on every vertex should be of the following 11 kinds of possibilities: $\star1$, $\star2$, $\star3$, $\star4$, $\star5$, $\star6$, $\star7$, $\star2'$, $\star7'$, $\star3'$, $\star6'$. Moreover, every two adjacent vertices should be of the types in (\ref{w4}) and (\ref{r22}).
\item[2.] Except $\star 1$, any other type of actions of $U_q(sl(2))$ on $A_q(3)$ can not be endowed with two different vertices simultaneously, since the relations (\ref{w2}) acting on $x_1$, $x_2$, $x_3$ producing zero can not be satisfied.
\item[3.]If every vertex in the Dynkin diagram of $sl(m+1)$ is endowed an action of $U_q(sl(2))$ on $A_q(3)$ which is not Case $\star1$, any two vertices which are not adjacent can not be endowed with the types which are adjacent (\ref{w4}) and (\ref{r22}).
\end{itemize}

\begin{theorem}\label{r33}
If $m\geq 4$, all module-algebra structures of $U_q(sl(m+1))$ on $A_q(3)$ are as follows
\begin{eqnarray*}
&&k_i(x_1)=\pm x_1, \qquad k_i(x_2)=\pm x_2,  \qquad k_i(x_3)=\pm x_3,\\
&&e_i(x_1)=e_i(x_2)=e_i(x_3)=f_i(x_1)=f_i(x_2)=f_i(x_3)=0,
\end{eqnarray*}
for any $i\in\{1,2,\cdots, m\}$.

For $m=3$, all module-algebra structures of $U_q(sl(4))$ on $A_q(3)$ are given by\\
(1)
\begin{eqnarray*}
&&k_i(x_1)=\pm x_1, \qquad k_i(x_2)=\pm x_2,  \qquad k_i(x_3)=\pm x_3,\\
&&e_i(x_1)=e_i(x_2)=e_i(x_3)=f_i(x_1)=f_i(x_2)=f_i(x_3)=0,
\end{eqnarray*}
for any $i\in\{1,2,3\}$. All these module-algebra structures are not pairwise non-isomorphic.\\
(2)\begin{eqnarray*}
&&k_i(x_1)=qx_1,~~k_i(x_2)=q^{-1}x_2,~~k_i(x_3)=x_3, \\
&&e_i(x_1)=0,~~e_i(x_2)=a_2x_1,~~e_i(x_3)=0,\\
&&f_i(x_1)=a_2^{-1}x_2,~~f_i(x_2)=0,~~f_i(x_3)=0,\\
&&k_j(x_1)=q^{-2}x_1,~~k_j(x_2)=q^{-1}x_2,~~k_j(x_3)=q^{-1}x_3,\\
&&e_j(x_1)=a_0,~~e_j(x_2)=0,~~e_j(x_3)=0,\\
&&f_j(x_1)=-qa_0^{-1}x_1^2,~~f_j(x_2)=-qa_0^{-1}x_1x_2,~~f_j(x_3)=-qa_0^{-1}x_1x_3,\\
&&k_s(x_1)=qx_1,~~k_s(x_2)=qx_2, k_s(x_3)=q^2x_3,\\
&&e_s(x_1)=-qf_0^{-1}x_1x_3,~~e_s(x_2)=-qf_0^{-1}x_2x_3,~~e_s(x_3)=-qf_0^{-1}x_3^2,\\
&&f_s(x_1)=0,~~f_s(x_2)=0,~~f_s(x_3)=f_0,
\end{eqnarray*}
where $a_2$, $a_0$, $f_0\in\mathbb{C}\backslash\{0\}$ and $i=1$, $j=2$, $k=3$ or $i=3$, $j=2$, $k=1$. All these module-algebra structures are isomorphic to that with $a_0=a_2=f_0=1$.\\
(3)\begin{eqnarray*}
&&k_i(x_1)=qx_1,~~k_i(x_2)=q^{-1}x_2,~~k_i(x_3)=x_3, \\
&&e_i(x_1)=0,~~e_i(x_2)=a_2x_1,~~e_i(x_3)=0,\\
&&f_i(x_1)=a_2^{-1}x_2,~~f_i(x_2)=0,~~f_i(x_3)=0,\\
&&k_j(x_1)=x_1,~~k_j(x_2)=qx_2,~~k_j(x_3)=q^{-1}x_3,\\
&&e_j(x_1)=0,~~e_j(x_2)=0,~~e_j(x_3)=b_3x_2,\\
&&f_j(x_1)=0,~~f_j(x_2)=b_3^{-1}x_3,~~f_j(x_3)=0,\\
&&k_s(x_1)=qx_1,~~k_s(x_2)=qx_2,~~k_s(x_3)=q^2x_3,\\
&&e_s(x_1)=-qf_0^{-1}x_1x_3,~~e_s(x_2)=-qf_0^{-1}x_2x_3,~~e_s(x_3)=-qf_0^{-1}x_3^2,\\
&&f_s(x_1)=0,~~f_s(x_2)=0,~~f_s(x_3)=f_0,
\end{eqnarray*}
where $a_2$, $b_3$, $f_0\in\mathbb{C}\backslash\{0\}$ and $i=1$, $j=2$, $s=3$ or $i=3$, $j=2$, $s=1$. All these module-algebra structures are isomorphic to that with $a_2=b_3=f_0=1$.\\
(4)\begin{eqnarray*}
&&k_i(x_1)=q^{-2}x_1,~~k_i(x_2)=q^{-1}x_2,~~k_i(x_3)=q^{-1}x_3,\\
&&e_i(x_1)=a_0,~~e_i(x_2)=0,~~e_i(x_3)=0,\\
&&f_i(x_1)=-qa_0^{-1}x_1^2,~~f_i(x_2)=-qa_0^{-1}x_1x_2,~~f_i(x_3)=-qa_0^{-1}x_1x_3,\\
&&k_j(x_1)=qx_1,~~k_j(x_2)=qx_2,~~k_j(x_3)=q^2x_3,\\
&&e_j(x_1)=-qf_0^{-1}x_1x_3,~~e_j(x_2)=-qf_0^{-1}x_2x_3,~~e_j(x_3)=-qf_0^{-1}x_3^2,\\
&&f_j(x_1)=0,~~f_j(x_2)=0,~~f_j(x_3)=f_0,\\
&&k_s(x_1)=x_1,~~k_s(x_2)=qx_2,~~k_s(x_3)=q^{-1}x_3,\\
&&e_s(x_1)=0,~~e_s(x_2)=0,~~e_s(x_3)=b_3x_2,\\
&&f_s(x_1)=0,~~f_s(x_2)=b_3^{-1}x_3,~~f_s(x_3)=0,
\end{eqnarray*}
where $b_3$, $a_0$, $f_0\in\mathbb{C}\backslash\{0\}$ and $i=1$, $j=2$, $k=3$ or $i=3$, $j=2$, $k=1$. All these module-algebra structures are isomorphic to that with $a_0=b_3=f_0=1$.\\
(5)\begin{eqnarray*}
&&k_i(x_1)=q^{-2}x_1,~~k_i(x_2)=q^{-1}x_2,~~k_i(x_3)=q^{-1}x_3,\\
&&e_i(x_1)=a_0,~~e_i(x_2)=0,~~e_i(x_3)=0,\\
&&f_i(x_1)=-qa_0^{-1}x_1^2,~~f_i(x_2)=-qa_0^{-1}x_1x_2,~~f_i(x_3)=-qa_0^{-1}x_1x_3,\\
&&k_j(x_1)=qx_1,~~k_j(x_2)=q^{-1}x_2,~~k_j(x_3)=x_3,\\
&&e_j(x_1)=0,~~e_j(x_2)=a_2x_1,~~e_j(x_3)=0,\\
&&f_j(x_1)=a_2^{-1}x_2,~~f_j(x_2)=0,~~f_j(x_3)=0,\\
&&k_s(x_1)=x_1,~~k_s(x_2)=qx_2,~~k_s(x_3)=q^{-1}x_3,\\
&&e_s(x_1)=0,~~e_s(x_2)=0,~~e_s(x_3)=b_3x_2,\\
&&f_s(x_1)=0,~~f_s(x_2)=b_3^{-1}x_3,~~f_s(x_3)=0,
\end{eqnarray*}
where $b_3$, $a_0$, $a_2\in\mathbb{C}\backslash\{0\}$ and $i=1$, $j=2$, $k=3$ or $i=3$, $j=2$, $k=1$. All these module-algebra structures are isomorphic to that with $a_0=a_2=b_3=1$.
\end{theorem}
\begin{proof}
 First, we consider the case when $m\geq 5$. By the above discussion, since there are no paths in (\ref{r22}) whose length is larger than $4$ and any two vertices which are not adjacent in this path have no edge connecting them in (\ref{w4}) and (\ref{r22}), the unique possibility of putting the actions of $U_q(sl(2))$ on the $m$ vertices in the Dynkin diagram is as follows:
$$\xymatrix{
{\star1} \ar@{-}[r]& {\star1} \ar@{-}[r]&{\cdots}\ar@{-}[r]& {\star1} \ar@{-}[r]&{\star1}}.$$
Obviously, the above case determines the module-algebra structures of $U_q(sl(m+1))$ on $A_q(3)$.

Second, let us study the case when $m=3$. By the above rules, and because
the Dynkin diagram of $sl(4)$ is symmetric, we only need to check the following cases
\begin{eqnarray*}
&&\xymatrix{
{\star7} \ar@{-}[r]& {\star4} \ar@{-}[r]& {\star5}},~~~\xymatrix{
{\star7} \ar@{-}[r]& {\star4} \ar@{-}[r]& {\star2}},~~\xymatrix{
{\star4} \ar@{-}[r]& {\star2} \ar@{-}[r]& {\star3}},\\
&&\xymatrix{
{\star4} \ar@{-}[r]& {\star5} \ar@{-}[r]& {\star3}},~~\xymatrix{
{\star4} \ar@{-}[r]& {\star5} \ar@{-}[r]& {\star6}},~~~\xymatrix{
{\star2} \ar@{-}[r]& {\star3} \ar@{-}[r]& {\star5}},\\
&&\xymatrix{
{\star2} \ar@{-}[r]& {\star4} \ar@{-}[r]& {\star5}},~~~\xymatrix{
{\star3} \ar@{-}[r]& {\star5} \ar@{-}[r]& {\star6}},~~\xymatrix{
{\star1} \ar@{-}[r]& {\star1} \ar@{-}[r]& {\star1}}.
\end{eqnarray*}
To determine the module-algebra structures of $U_q(sl(4))$ on $A_q(3)$, we still have to check the following equalities
\begin{eqnarray*}
&&k_1e_3(u)=e_3k_1(u), k_1f_3(u)=f_3k_1(u), \\
&&k_3e_1(u)=e_1k_3(u), k_3f_1(u)=f_1k_3(u), \\
&&e_1f_3(u)=f_3e_1(u), e_3f_1(u)=f_1e_3(u),\\
&&e_1e_3(u)=e_3e_1(u), f_1f_3(u)=f_3f_1(u),
\end{eqnarray*}
for any $u\in\{x_1,x_2,x_3\}$. For $\xymatrix{
{\star7} \ar@{-}[r]& {\star4} \ar@{-}[r]& {\star5}}$, since $k_1e_3(z)=k_1(b_3y)=q^2b_3y$ and $e_3k_1(z)=q^{-1}b_3y$, $k_1e_3(z)\neq e_3k_1(z)$. Therefore, $\xymatrix{
{\star7} \ar@{-}[r]& {\star4} \ar@{-}[r]& {\star5}}$ is excluded. Similarly, we exclude $\xymatrix{
{\star7} \ar@{-}[r]& {\star4} \ar@{-}[r]& {\star2}}$, $\xymatrix{
{\star4} \ar@{-}[r]& {\star5} \ar@{-}[r]& {\star6}}$, and $\xymatrix{
{\star3} \ar@{-}[r]& {\star5} \ar@{-}[r]& {\star6}}$. Moreover, it is easy to check the five remaining cases determine the module-algebra structures of $U_q(sl(4))$ on $A_q(3)$.

Thirdly, we consider the case when $m=4$. By the discussion above, we only need to check the cases
\begin{eqnarray*}
&&\xymatrix{
{\star7} \ar@{-}[r]& {\star4} \ar@{-}[r]&{\star2} \ar@{-}[r]& {\star3}},~~\xymatrix{
{\star7} \ar@{-}[r]& {\star4} \ar@{-}[r]&{\star5} \ar@{-}[r]& {\star3}},\\
&&\xymatrix{
{\star7} \ar@{-}[r]& {\star4} \ar@{-}[r]&{\star5} \ar@{-}[r]& {\star6}},~~\xymatrix{
{\star2} \ar@{-}[r]& {\star3} \ar@{-}[r]&{\star5} \ar@{-}[r]& {\star6}},\\
&&\xymatrix{
{\star1} \ar@{-}[r]& {\star1} \ar@{-}[r]&{\star1} \ar@{-}[r]& {\star1}}.
\end{eqnarray*}
Since the three adjacent vertices in the Dynkin diagram of $sl(5)$ corresponds to one Hopf algebra isomorphic to $U_q(sl(4))$, by the results of the module-algebra structures of $U_q(sl(4))$ on $A_q(3)$, there is only one possibility:
$$\xymatrix{
{\star1} \ar@{-}[r]& {\star1} \ar@{-}[r]&{\star1} \ar@{-}[r]& {\star1}}.$$

Finally, we consider the isomorphism classes. Here, we only show that all module-algebra structures of $U_q(sl(4))$ on $A_q(3)$ in Case (2) are isomorphic to that with $a_0=a_2=f_0=1$. The desired isomorphism is given by $\psi_{a_0, a_2, f_0}:x_1\rightarrow a_0x_1, x_2\rightarrow a_0a_2x_2, x_3\rightarrow f_0x_3$. Other cases can be considered similarly.
\end{proof}
\begin{remark}
By Theorem \ref{r33}, the classical limits of the above actions,
 i.e., the Lie algebra $sl_{m+1}$-actions
 by differentiations on $\mathbb{C}[x_1,x_2,x_3]$ can also be obtained
 as before.
\end{remark}
\section{Structures of $U_q(sl(m+1))$-symmetries on $A_q(n)$}
In this section, we will study the module-algebra structures of $U_q(sl(m+1))$ on $A_q(n)$.

We also consider the module-algebra structures of $U_q(sl(2))$ on $A_q(n)$ first. Let
\begin{eqnarray}
M\stackrel{def}{=}\left[\begin{array}{cccc}
k(x_1)& k(x_2)& \cdots& k(x_n)\\
e(x_1)& e(x_2)& \cdots& e(x_n)\\
f(x_1)& f(x_2)& \cdots& f(x_n)\end{array}\right].
\end{eqnarray}

As usual, we can set
\begin{eqnarray*}
M_k\stackrel{def}{=}\left[\begin{array}{cccc}
k(x_1)& k(x_2)& \cdots& k(x_n)\end{array}\right]= \left[\begin{array}{cccc}
\alpha_1 x_1&\alpha_2 x_2&\cdots&\alpha_n x_n\end{array}\right],
\end{eqnarray*}
where $\alpha_i$ for $i\in\{1,\cdots, n\}$ are non-zero complex numbers. So, every monomial $x_1^{m_1}x_2^{m_2}\cdots x_n^{m_n}\in A_q(n)$ is an eigenvector for $k$ and the
associated eigenvalue $\alpha_1^{m_1}\alpha_2^{m_2}\cdots \alpha_n^{m_n}$ is called the \emph{weight} of this monomial, which will be written as $wt(x_1^{m_1}x_2^{m_2}\cdots x_n^{m_n})=\alpha_1^{m_1}\alpha_2^{m_2}\cdots \alpha_n^{m_n}$.

Let
\begin{eqnarray}
M_{ef}\stackrel{def}{=}\left[\begin{array}{cccc}
e(x_1)& e(x_2)& \cdots & e(x_n)\\
f(x_1)& f(x_2)&\cdots &f(x_n) \end{array}\right].
\end{eqnarray}
Then, we have
\begin{eqnarray*}
wt(M_{ef})\stackrel{def}{=}\left[\begin{array}{cccc}
wt(e(x_1))&wt(e(x_2))& \cdots& wt(e(x_n))\\
wt(f(x_1))& wt(f(x_2))& \cdots& wt(f(x_n)) \end{array}\right]\end{eqnarray*}
\begin{eqnarray*}\bowtie \left[\begin{array}{cccc}
q^2\alpha_1& q^2\alpha_2&\cdots& q^2\alpha_n\\
q^{-2}\alpha_1&q^{-2}\alpha_2&\cdots&q^{-2}\alpha_n\end{array}\right].
\end{eqnarray*}

Set \begin{eqnarray*}
(M)_0=\left[\begin{array}{cccc}
0& 0 &\cdots & 0\\
a_1&a_2&\cdots &a_n\\
b_1&b_2&\cdots &b_n \end{array}\right]_0.
\end{eqnarray*}
Then, we obtain
\begin{eqnarray}\label{e64}
wt((M)_0)
\bowtie \left[\begin{array}{cccc}
0& 0& \cdots & 0\\
q^2\alpha_1 & q^2\alpha_2&\cdots &q^2\alpha_n\\
q^{-2}\alpha_1&q^{-2}\alpha_2&\cdots &q^{-2}\alpha_n \end{array}\right]_0\end{eqnarray}
\begin{eqnarray*}\bowtie \left[\begin{array}{cccc}
0& 0& \cdots & 0\\
1& 1&\cdots&1\\
1& 1 & \cdots &1 \end{array}\right]_0 .
\end{eqnarray*}
An application of $e$ and $f$ to (\ref{e61}) gives the following equalities
\begin{eqnarray}\label{e62}
x_ie(x_j)-q\alpha_ie(x_j)x_i=qx_je(x_i)-\alpha_je(x_i)x_j~~for~~i>j,\end{eqnarray}
\begin{eqnarray}\label{e63}f(x_i)x_j-q\alpha_j^{-1}x_jf(x_i)=qf(x_j)x_i-\alpha_i^{-1}x_if(x_j)~~for~~i>j.\end{eqnarray}

After projecting the equalities above to $(A_q(n))_1$, we obtain
\begin{eqnarray*}
a_j(1-q\alpha_i)x_i=a_i(q-\alpha_j)x_j~~for~~i>j,\\
b_i(1-q\alpha_j^{-1})x_j=b_j(q-\alpha_i^{-1})x_i~~for~~i>j.
\end{eqnarray*}
Therefore, for $i>j$, we obtain
\begin{eqnarray}
&&a_j\neq 0 \Rightarrow \alpha_i=q^{-1},\qquad a_i\neq 0 \Rightarrow \alpha_j=q,\\
&&b_i\neq 0 \Rightarrow \alpha_j=q,\qquad b_j\neq 0 \Rightarrow \alpha_i=q^{-1},
\end{eqnarray}
Then, we have that for any $j\in\{1,\cdots,n\}$,
\begin{eqnarray}
a_j\neq 0 \Rightarrow \alpha_i=q^{-1}~~for~~\forall~~ i>j,~~\alpha_i=q~~for~~\forall~~i<j, \\
b_j\neq 0 \Rightarrow \alpha_i=q^{-1}~~for~~\forall~~ i>j,~~\alpha_i=q~~for~~\forall~~i<j.
\end{eqnarray}
By (\ref{e64}) and using the above equalities, we get
\begin{eqnarray*}
a_j\neq 0 \Rightarrow \alpha_1=q, \cdots, \alpha_{j-1}=q, \alpha_j=q^{-2}, \alpha_{j+1}=q^{-1}, \cdots, \alpha_n=q^{-1},\\
b_j\neq 0 \Rightarrow \alpha_1=q, \cdots, \alpha_{j-1}=q, \alpha_j=q^{2}, \alpha_{j+1}=q^{-1}, \cdots, \alpha_n=q^{-1}.
\end{eqnarray*}
So, there are $2n+1$ cases as follows: $a_j\neq 0$, $a_i=0$ for $i\neq j$ and all $b_i=0$ for any $j\in \{1,\cdots, n\}$;
$b_j\neq 0$, $b_i=0$ for $i\neq j$ and all $a_i=0$ for any $j\in \{1,\cdots, n\}$; $a_j=0$ and $b_j=0$ for any $j\in \{1,\cdots, n\}$.

For the 1-st homogeneous component, since $wt(e(x_i))=q^2wt(x_i)\neq wt(x_i)$, we have $(e(x_i))_1=\sum_{s\neq i}c_{is}x_s$ for some $c_{is}\in \mathbb{C}$. Similarly, we set $(f(x_i))_1=\sum_{s\neq i}d_{is}x_s$ for some $d_{is}\in \mathbb{C}$.

After projecting Equations (\ref{e62})-(\ref{e63}) to $(A_q(n))_2$, we can obtain that for any $i>j$,
\begin{eqnarray*}
\sum_{\substack{s\neq j\\
s<i}}(q-q\alpha_i)c_{js}x_sx_i+(1-q\alpha_i)c_{ji}x_i^2+\sum_{s>i}(1-q^2\alpha_i)c_{js}x_ix_s=\\
\sum_{s<j}(q^2-\alpha_j)c_{is}x_sx_j+(q-\alpha_j)c_{ij}x_j^2+\sum_{\substack{s\neq i\\
s>j}}(q-q\alpha_j)c_{is}x_jx_s,\end{eqnarray*}
\begin{eqnarray*}
\sum_{s<j}(1-q^2\alpha_j^{-1})d_{is}x_sx_j+(1-q\alpha_j^{-1})d_{ij}x_j^2+\sum_{\substack{s>j\\
s\neq i}}(q-q\alpha_j^{-1})d_{is}x_jx_s=\\
\sum_{\substack{s<i\\s\neq j}}(q-q\alpha_i^{-1})d_{js}x_sx_i+(q-\alpha_i^{-1})d_{ji}x_i^2+\sum_{\substack{s>i}}(q^2-\alpha_i^{-1})d_{js}x_ix_s. \end{eqnarray*}
Therefore, we have
\begin{eqnarray*}
c_{js}\neq 0~~(s<i,~~s\neq j)\Rightarrow \alpha_i=1, ~~c_{js}=0~~for~~all~~s\geq i,\end{eqnarray*}
\begin{eqnarray*}
c_{ji}\neq 0\Rightarrow \alpha_i=q^{-1},~~c_{js}=0~~for ~~any~~s\neq i,\end{eqnarray*}
\begin{eqnarray*}
c_{js}\neq 0~~ (s>i)\Rightarrow \alpha_i=q^{-2}, ~~c_{js}=0~~for~~all~~s\leq i,\end{eqnarray*}
\begin{eqnarray*}c_{is}\neq 0 ~~(s<j)\Rightarrow \alpha_j=q^2, ~~c_{is}=0~~for~~all~~s\geq j,\end{eqnarray*}
\begin{eqnarray*}c_{ij}\neq 0 \Rightarrow \alpha_j=q, ~~c_{is}=0~~for~~all~~s\neq j,\end{eqnarray*}
\begin{eqnarray*}c_{is}\neq 0 ~~(s>j,~~s\neq i)\Rightarrow \alpha_j=1, ~~c_{is}=0~~for~~all~~s\leq j,\end{eqnarray*}
\begin{eqnarray*}d_{is}\neq 0 ~~(s<j)\Rightarrow \alpha_j=q^2, ~~d_{is}=0~~for~~all~~s>j,\end{eqnarray*}
\begin{eqnarray*}d_{ij}\neq 0 \Rightarrow \alpha_j=q, ~~d_{is}=0~~for~~all~~s\neq j,\end{eqnarray*}
\begin{eqnarray*}d_{is}\neq 0 ~~(s>j,~~s\neq i)\Rightarrow \alpha_j=1, ~~d_{is}=0~~for~~all~~s\leq j,\end{eqnarray*}
\begin{eqnarray*}d_{js}\neq 0 ~~ (s<i,~~s\neq j) \Rightarrow \alpha_i=1,~~d_{js}=0~~for~~all~~s\geq i,\end{eqnarray*}
\begin{eqnarray*}d_{ji}\neq 0 \Rightarrow \alpha_i=q^{-1},~~d_{js}=0~~for~~all~~s\neq i,\end{eqnarray*}
\begin{eqnarray*}d_{js}\neq 0 ~~(s>i) \Rightarrow \alpha_i=q^{-2},~~d_{js}=0~~for~~all~~s\leq i.
\end{eqnarray*}

Therefore, we have that for any $j\in \{1,\cdots, n\}$,
\begin{eqnarray*}
c_{js}\neq 0 ~~(s>j) \Rightarrow \alpha_1=1,\cdots, \alpha_{j-1}=1,
 \alpha_{j+1}=q^{-2},\cdots,\\ \alpha_{s-1}=q^{-2},\alpha_s=q^{-1},\alpha_{s+1}=1,\cdots, \alpha_n=1,\end{eqnarray*}
\begin{eqnarray*}
c_{js}\neq 0 ~~(s<j) \Rightarrow \alpha_1=1,\cdots, \alpha_{s-1}=1, \alpha_s=q, \alpha_{s+1}=q^{2},\cdots, \\ \alpha_{j-1}=q^{2},\alpha_{j+1}=1,\cdots, \alpha_n=1,\end{eqnarray*}
\begin{eqnarray*}
d_{js}\neq 0 ~~(s>j) \Rightarrow \alpha_1=1,\cdots, \alpha_{j-1}=1, \alpha_{j+1}=q^{-2},\cdots,\\ \alpha_{s-1}=q^{-2},\alpha_s=q^{-1},\alpha_{s+1}=1,\cdots, \alpha_n=1,\end{eqnarray*}
\begin{eqnarray*}
d_{js}\neq 0 ~~(s<j) \Rightarrow \alpha_1=1,\cdots, \alpha_{s-1}=1, \alpha_s=q, \alpha_{s+1}=q^{2},\cdots, \\ \alpha_{j-1}=q^{2},\alpha_{j+1}=1,\cdots, \alpha_n=1.\end{eqnarray*}

Since
$
wt((M_{ef})_1)=\left[\begin{array}{cccc}
q^2\alpha_1 &q^2\alpha_2&\cdots &q^2\alpha_n\\
q^{-2}\alpha_1&q^{-2}\alpha_2&\cdots&q^{-2}\alpha_n\end{array}\right]_1$, we obtain that
for any $j\in \{1,\cdots, n\}$,
\begin{eqnarray*}
c_{js}\neq 0 ~~(s>j) \Rightarrow \alpha_1=1,\cdots, \alpha_{j-1}=1, \alpha_j=q^{-3},
 \alpha_{j+1}=q^{-2},\cdots,\\ \alpha_{s-1}=q^{-2},\alpha_s=q^{-1},\alpha_{s+1}=1,\cdots, \alpha_n=1,\end{eqnarray*}
\begin{eqnarray*}
c_{js}\neq 0 ~~(s<j) \Rightarrow \alpha_1=1,\cdots, \alpha_{s-1}=1, \alpha_s=q, \alpha_{s+1}=q^{2},\cdots, \\ \alpha_{j-1}=q^{2},\alpha_j=q^{-1},\alpha_{j+1}=1,\cdots, \alpha_n=1,\end{eqnarray*}
\begin{eqnarray*}
d_{js}\neq 0 ~~(s>j) \Rightarrow \alpha_1=1,\cdots, \alpha_{j-1}=1, \alpha_j=q,\alpha_{j+1}=q^{-2},\cdots,\\ \alpha_{s-1}=q^{-2},\alpha_s=q^{-1},\alpha_{s+1}=1,\cdots, \alpha_n=1,\end{eqnarray*}
\begin{eqnarray*}
d_{js}\neq 0 ~~(s<j) \Rightarrow \alpha_1=1,\cdots, \alpha_{s-1}=1, \alpha_s=q, \alpha_{s+1}=q^{2},\cdots, \\ \alpha_{j-1}=q^{2},\alpha_j=q^3,\alpha_{j+1}=1,\cdots, \alpha_n=1.\end{eqnarray*}

By the above discussion, we have only the following possibilities for the 1-st homogeneous component:
$c_{ij}\neq 0$ for some $i\neq j$, other $c_{st}$ equal to zero and all $d_{st}=0$; $d_{ij}\neq 0$ for some $i\neq j$, other $d_{st}$ equal to zero and all $c_{st}=0$; $c_{j+1,j}\neq 0$, $d_{j,j+1}\neq 0$ for some $j\in \{1,\cdots, n\}$.

Obviously, if both the 0-th homogeneous component and the 1-th homogeneous component of $M_{ef}$ are nonzero, there are no possibilities except
when $n=3$. But, the case when $n=3$ has been considered in Section 4. Therefore, we assume that $n\geq 4$ in the following paper.

Moreover, there are no possibilities when the 0-th homogeneous component of $M_{ef}$ is 0 and the 1-th homogeneous component of $M_{ef}$ has only one nonzero position.
The reasons are the same as those in \cite{dup/sin3}.

Therefore, we have to consider the following cases\\
 $\left(\left[\begin{array}{cccccc}
0&0 &\cdots &a_i &\cdots & 0\\
0& 0 &\cdots &0 &\cdots &0\end{array}\right]_0, \left[\begin{array}{cccc}
0 & 0& \cdots &0\\
0 &  0 &\cdots&0 \end{array}\right]_1\right)$ for any $i\in \{1,\cdots n\}$,\\
$\left(\left[\begin{array}{cccccc}
0& 0 &\cdots & 0 & \cdots& 0\\
0 & 0 & \cdots &b_i &\cdots & 0\end{array}\right]_0, \left[\begin{array}{cccc}
0 & 0&\cdots & 0\\
0 &0 &\cdots & 0 \end{array}\right]_1\right)$ for any $i\in \{1,\cdots n\}$,\\
$\left(\left[\begin{array}{cccc}
0 &0&\cdots &0\\
0 & 0 &\cdots &0 \end{array}\right]_0, \left[\begin{array}{ccccccc}
0 & 0& \cdots &0 & c_{j+1,j}& \cdots & 0\\
0 & 0 & \cdots &d_{j,j+1} & 0 &\cdots & 0 \end{array}\right]_1\right)$ for any $j\in$$\{1,$$\cdots,$\\$n$$-1\}$,\\ $\left(\left[\begin{array}{cccc}
0 &0&\cdots &0\\
0 & 0 &\cdots &0 \end{array}\right]_0, \left[\begin{array}{cccc}
0 & 0& \cdots &0\\
0& 0 &\cdots & 0 \end{array}\right]_1\right)$.

By the above cases and the discussions in Section 4, we can obtain the following proposition.
\begin{proposition}\label{sss1}
For $n\geq 4$, there are the module-algebra structures of $U_q(sl(2))$ on $A_q(n)$ as follows:\\
(1) $$k(x_i)=\pm x_i,~~e(x_i)=f(x_i)=0,$$ for any $i\in\{1,\cdots, n\}$. All these structures are pairwise nonisomorphic.\\
(2) $$k(x_i)=qx_i~~for~~\forall~~i<j,~~k(x_j)=q^{-2}x_j,~~k(x_i)=q^{-1}x_i~~for~~\forall~~i>j,$$
$$e(x_i)=0~~for ~~\forall~~i\neq j,~~e(x_j)=a_j,$$
$$f(x_i)=a_j^{-1}x_ix_j ~~for~~\forall~~i<j,~~f(x_j)=-qa_j^{-1}x_j^2,$$
$$f(x_i)=-qa_j^{-1}x_jx_i~~~~for~~\forall~~i>j,$$
for any $j\in \{1,\cdots, n\}$ and $a_j \in \mathbb{C}\setminus\{0\}$. If $j$ is fixed, then all these structures are isomorphic to that with
$a_j=1$. \\
(3) $$k(x_i)=qx_i~~for~~\forall~~i<j,~~k(x_j)=q^{2}x_j,~~k(x_i)=q^{-1}x_i~~for~~\forall~~i>j,$$
$$e(x_i)=-qb_j^{-1}x_ix_j ~~for~~\forall~~i<j,~~e(x_j)=-qb_j^{-1}x_j^2,$$
$$e(x_i)=b_j^{-1}x_jx_i~~~~for~~\forall~~i>j,$$
$$f(x_i)=0~~for ~~\forall~~i\neq j,~~f(x_j)=b_j,$$
for any $j\in \{1,\cdots, n\}$ and $b_j \in \mathbb{C}\setminus\{0\}$. If $j$ is fixed, then all these structures are isomorphic to that with
$b_j=1$.\\
(4)  $$k(x_i)=x_i~~for~~\forall~~i<j,~~k(x_j)=qx_j,~~k(x_{j+1})=q^{-1}x_{j+1},~~k(x_i)=x_i~~for~~\forall~~i>j+1,$$
$$e(x_i)=0 ~~for~~\forall~~i\neq {j+1},~~e(x_{j+1})=c_{j+1,j}x_j,$$
$$f(x_i)=0~~for ~~\forall~~i\neq j,~~f(x_j)=c_{j+1,j}^{-1} x_{j+1},$$
for any $j\in \{1,\cdots, n-1\}$ and $c_{j+1,j} \in \mathbb{C}\setminus\{0\}$. If $j$ is fixed, then all these structures are isomorphic to that with $c_{j+1,j}=1$.
\end{proposition}
\begin{remark}
In Proposition \ref{sss1} we present only the simplest module-algebra structures. It is also complicated to give the solutions of (\ref{e62}) and (\ref{e63}) for all cases. For example, by a very complex computation, we can obtain that in Case $\left(\left[\begin{array}{cccc}
a_1 &0 &\cdots& 0\\
0 & 0 & \cdots& 0\end{array}\right]_0, \left[\begin{array}{cccc}
0 & 0 &\cdots & 0\\
0& 0 &\cdots & 0 \end{array}\right]_1\right)$, all $U_q(sl(2))$-module algebra structures on $A_q(n)$ are given by
\normalsize{
 $$k(x_1)=q^{-2}x_1,~~~k(x_i)=q^{-1}x_i~~for~~\forall~~i>1,$$
 $$e(x_1)=a_1,~~~e(x_i)=0 ~~for~~\forall~~i>1,$$
  $$f(x_1)=-qa_1^{-1}x_1^2,$$
\begin{eqnarray*}
f(x_2)&=&-qa_1^{-1}x_1x_2+\sum_{2<s\leq n}\widehat{v}_{2s2}x_2x_s^2+\sum_{2<s<k<l\leq n}\alpha_{22kl}x_2x_kx_l+\beta_{22}x_2^3,\end{eqnarray*}
\begin{eqnarray*}
f(x_i)&=&-qa_1^{-1}x_1x_i+(3)_q\beta_{22}x_2^2x_i-\sum_{2<s<i\leq n}q^{-1}(3)_q\widehat{v_{2s2}}x_s^2x_i\\
&&+\frac{(2)_q}{(3)_q}\alpha_{n2in}x_2x_i^2+\sum_{2<i<t\leq n}\widehat{v_{2t2}}x_ix_t^2-\sum_{2<s<i<n}q^{-1}(2)_q\alpha_{22si}\\
&&\cdot x_sx_i^2+\frac{q}{(2)_q}\widehat{v_{nn2}}x_2x_ix_n+
\sum_{2<k<i<n}\alpha_{n2kn}x_2x_kx_i\\
&&+\sum_{2<i<k<n}\frac{q}{(3)_q}\alpha_{n2kn}x_2x_ix_k-\sum_{2<s<i<k\leq n}\alpha_{22sk}x_sx_ix_k\\
&&+
\sum_{2<i<k<l\leq n}\alpha_{22kl}x_ix_kx_l-\sum_{2<s<k<i<n}q^{-1}(3)_q\alpha_{22sk}x_sx_kx_i\\
&&-q^{-1}\widehat{v_{2i2}}x_i^3,
\end{eqnarray*}
where $2<i<n$,
\begin{eqnarray*}
f(x_n)&=&-qa_1^{-1}x_1x_n+(3)_q\beta_{22}x_2^2x_n-\sum_{2<s<n}q^{-1}(3)_q\widehat{v_{2s2}}x_s^2x_n\\
\\&&+\widehat{v_{nn2}}x_2x_n^2-\sum_{2<k<n}q^{-1}(2)_q\alpha_{22kn}x_kx_n^2+\sum_{2<k< n}\alpha_{n2kn}x_2x_kx_n\\
&&-\sum_{2<s<k<n}q^{-1}(3)_q\alpha_{22sk}x_sx_kx_n-q^{-1}\widehat{v_{2n2}}x_n^3,
\end{eqnarray*}}
where $a_1\in \mathbb{C}\setminus\{0\}$ and $\widehat{v_{2i2}}$, $\alpha_{22kl}$, $\beta_{22}$, $\widehat{v_{nn2}}$, $\alpha_{n2kn}\in \mathbb{C}$.
\end{remark}

Let us denote the module-algebra structures of Case (1), those in Case (2), Case (3) and Case (4) in Proposition \ref{sss1} by $D$, $A_j$, $B_j$ and $C_j$ respectively. For determining the module-algebra structures of $U_q(sl(3))$ on $A_q(n)$, we only need to check whether (\ref{e11})-(\ref{e12})
hold for any $u\in\{x_1,\cdots,x_n\}$. For convenience, we introduce a notation: if the actions of $k_s$, $e_s$, $f_s$ are of the type $A_i$ and the actions of $k_t$, $e_t$, $f_t$ are of the type $B_j$, they determine a module-algebra structure of
$U_q(sl(3))$ on $A_q(n)$ for $s=1$, $t=2$ or $s=2$, $t=1$, then we say $A_i$ and $B_j$ are compatible. By some computations, we can obtain that
$D$ and $D$ are compatible, $A_i$ and $B_j$ are compatible if and only if $i=1$ and $j=n$, $A_i$ and $C_j$ are compatible if and only if $i=j$,
$B_i$ and $C_j$ are compatible if and only if $j=i+1$, $C_i$ and $C_j$ are compatible if and only if $i=j+1$ or $i=j-1$, and any two other cases are not compatible.

Therefore, by the above discussion, similar to that in Section 4, we can obtain the following proposition.
\begin{proposition}
For $n\geq 4$, there are the module-algebra structures of $U_q(sl(3))$ on $A_q(n)$ as follows
\begin{eqnarray}
&&\xymatrix{{D}\ar@{-}[r]& {D}},\end{eqnarray}
\begin{eqnarray}
\label{r26}&&\xymatrix{ {A_1}\ar@{-}[ddrrrrr]\ar@{-}[d]& {A_2}\ar@{-}[d]& \cdots & {A_{n-2}}\ar@{-}[d] & {A_{n-1}}\ar@{-}[d] & \\
{C_1}\ar@{-}[r]\ar@{-}[dr]&{C_2}\ar@{-}[dr]\ar@{-}[r]&{\cdots}\ar@{-}[dr]\ar@{-}[r]&{C_{n-2}}\ar@{-}[dr]\ar@{-}[r]& {C_{n-1}}\ar@{-}[dr]&\\
& {B_2}& {\cdots}& {B_{n-2}} & {B_{n-1}} & {B_n} }.
\end{eqnarray}
Here, every two adjacent vertices determines two classes of module-algebra structures of $U_q(sl(3))$ on $A_q(n)$. For example, $\xymatrix{ {A_1}\ar@{-}[r]& {C_1}}$ corresponds to the following two kinds of module-algebra structures of $U_q(sl(3))$ on $A_q(n)$:
one is that the actions of $k_1$, $e_1$, $f_1$ are of type $A_1$ and the actions of $k_2$, $e_2$, $f_2$ are of type $C_1$; the other is that  the actions of $k_1$, $e_1$, $f_1$ are of type $C_1$ and the actions of $k_2$, $e_2$, $f_2$ are of type $A_1$.
\end{proposition}

Then, for determining the module-algebra structures of $U_q(sl(m+1))$ on $A_q(n)$, we have to find the pairs of vertices which are not adjacent in (\ref{r26}) and satisfy the following relation: $k_ie_j(x_s)=e_jk_i(x_s)$, $k_je_i(x_s)=e_ik_j(x_s)$, $k_if_j(x_s)=f_jk_i(x_s)$, $k_jf_i(x_s)=f_ik_j(x_s)$, $e_ie_j(x_s)=e_je_i(x_i)$, $e_if_j(x_s)=f_je_i(x_s)$,
$f_if_j(x_s)=f_jf_i(x_s)$ where one vertex corresponds to the actions of $k_i$, $e_i$ and $f_i$ and the other vertex corresponds to the actions of $k_j$, $e_j$ and $f_j$, $s\in\{1,\cdots, n\}$. It is easy to check that $A_i$ and $C_j$ satisfy the above relations if and only if $i<j$ or $i>j+1$, $B_i$ and $C_j$ satisfy the above relations if and only if $i<j$ or $i>j+1$, $C_i$ and $C_j$ satisfy the above relations if and only if $i\neq j+1$ or $j\neq i+1$, and any other two vertices do not satisfy the above relations.

Therefore, we obtain the following theorem.
\begin{theorem}
For $m\geq 3$, $n\geq 4$, there are the module-algebra structures of $U_q(sl(m+1))$ on $A_q(n)$ as follows:
\begin{eqnarray}
\label{r30}&&\underbrace{\xymatrix{{D}\ar@{-}[r]& {D}\ar@{-}[r]& {\cdots}\ar@{-}[r]& {D}}}_{m},\\
&&\xymatrix{{A_i}\ar@{-}[r]& {C_i}\ar@{-}[r]& {\cdots}\ar@{-}[r]& {C_{i+m-2}}},\\
&&\xymatrix{{C_i}\ar@{-}[r]& {C_{i+1}}\ar@{-}[r]& {\cdots}\ar@{-}[r]& {C_{i+m-2}}\ar@{-}[r]&B_{i+m-1}},\\
\label{r31}&&\xymatrix{{C_i}\ar@{-}[r]& {C_{i+1}}\ar@{-}[r]& {\cdots}\ar@{-}[r]& {C_{i+m-1}}},
\end{eqnarray}
\begin{eqnarray}
\xymatrix{{A_1}\ar@{-}[r]& {B_{n}}\ar@{-}[r]& {C_{n-1}}\ar@{-}[r]&{\cdots}\ar@{-}[r]& {C_{n+2-m}}},
\end{eqnarray}
where $n+2-m>1$,
\begin{eqnarray}\label{r27}
\xymatrix{{B_n}\ar@{-}[r]& {A_{1}}\ar@{-}[r]& {C_{1}}\ar@{-}[r]&{\cdots}\ar@{-}[r]& {C_{m-2}}},
\end{eqnarray}
where $m-2< n-1$.

Here, every such diagram corresponds to two classes of the module-algebra structures of $U_q(sl(m+1))$ on $A_q(n)$. 
For instance, there are the two module-algebra structures of $U_q(sl(m+1))$ on $A_q(n)$ corresponding to (\ref{r27}): 
the one is that the actions of $k_1$, $e_1$, $f_1$ are of the type $B_n$, 
those of $k_2$, $e_2$, $f_2$ are of the type $A_1$ and those of $k_i$, $e_i$, $f_i$ are of the type $C_{i-2}$ for any $3\leq i\leq m$. 
The other is that the actions of $k_i$, $e_i$, $f_i$ are of the type $C_{m-1-i}$ for any $1\leq i\leq m-2$, those of $k_{m-1}$, $e_{m-1}$, $f_{m-1}$ are of the type $A_1$ and those of $k_m$, $e_m$, $f_m$ are of the type $B_n$.
\end{theorem}

\begin{remark}
When $m=n-1$ and the indexes of the vertices of the Dynkin diagram are given $1$, $\cdots$, $n-1$ from the left to the right, the actions corresponds to (\ref{r31}), i.e., $$\xymatrix{{C_1}\ar@{-}[r]& {C_{2}}\ar@{-}[r]& {\cdots}\ar@{-}[r]& {C_{n-1}}},$$ is the case discussed in \cite{hu2000}. 
In addition, we are sure that when $m\geq n+1$, all the module-algebra structures of $U_q(sl(m+1))$ on $A_q(n)$ are of the type in
(\ref{r30}), since there are no paths whose length is larger than $n+1$ and any two vertices which are not adjacent in this path have no edge connecting them in (\ref{r26}). The detailed proof may be similar to that in Section 4. 
Moreover, the module-algebra structures of the quantum enveloping algebras corresponding to the other semisimple Lie algebras on $A_q(n)$ can be considered in the same way.
\end{remark}

{\bf Acknowledgments}.
{\em The first author (S.D.) is grateful to L.~Carbone and E.~Karolinsky for
fruitful discussions.
The second and third authors express their
gratitude to the support from the projects of the National Natural Science Foundation of 
China (No.11271318, No.11171296 and No. J1210038) and the Specialized Research Fund for the Doctoral Program of Higher Education of China (No. 20110101110010)}.

\end{document}